\documentclass{article}
\usepackage{amsfonts,amsmath, amssymb,latexsym,comment, enumerate,mathrsfs}
\usepackage[OT2,OT1]{fontenc}
\def\SH{\mbox{\fontencoding{OT2}\selectfont\char88}}
\usepackage{multirow}
\def\Z{{\mathbb Z}}

\def\SL{{\rm SL}}
\def\s{{\rm s}}
\def\sm{{\rm sm}}
\def\odd{{\rm odd}}
\def\even{{\rm even}}
\def\an{{\rm an}}

\def\Tr{{\rm Tr}}

\def\min{{\rm min}}
\def\sk{{\rm sk}}
\def\ind{{\rm ind}}

\def\GL{{\rm GL}}
\def\rad{{\rm rad}}

\def\PGL{{\rm PGL}}

\def\ca{{\check{a}}}
\def\cb{{\check{b}}}

\def\sf{{\rm sf}}
\def\sf{{\rm sf}}

\def\Stab{{\rm Stab}}
\def\Sel{{\rm Sel}}

\def\Sym{{\rm Sym}}

\def\Cl{{\rm Cl}}

\def\E{{\mathcal E}}
\newcommand{\cO}{{\mathcal O}}

\def\CS{{\mathcal S}}
\def\B{{\mathcal B}}
\def\P{{\mathbb P}}
\def\I{{\mathrm I}}
\def\Disc{{\rm Disc}}

\def\irr{{\rm irr}}

\def\red{{\rm red}}

\def\Res{{\rm Res}}
\def\Vol{{\rm Vol}}
\def\R{{\mathbb R}}
\def\F{{\mathbb F}}

\renewcommand{\L}{\mathcal{L}}

\def\rV{{\mathrm V}}

\def\FF{{\mathcal F}}

\def\GG{{\mathcal G}}
\def\RR{{\mathcal R}}
\def\Q{{\mathbb Q}}

\def\Z{{\mathbb Z}}
\def\G{{\mathbb G}}
\def\P{{\mathbb P}}
\def\F{{\mathbb F}}
\def\Q{{\mathbb Q}}
\def\C{{\mathbb C}}

\def\s{{\mathcal S}}
\def\cc{{\check{c}}}

\def\fz1{{F_{\Z,1}}}

\def\max{{\rm max}}

\newtheorem{thm}{Theorem}[section]
\newtheorem{theorem}{Theorem}[section]
\newtheorem{corollary}[theorem]{Corollary}
\newtheorem{cor}[thm]{Corollary}
\newtheorem{lemma}[theorem]{Lemma}

\newtheorem{proposition}[theorem]{Proposition}

\newenvironment{proof}{\noindent {\bf Proof:}}{$\Box$ \vspace{2 ex}}

\usepackage{xcolor,color,url,lmodern}

\title{Families of elliptic curves ordered by conductor}

\author{Ananth N.\ Shankar and Arul Shankar and Xiaoheng Wang}

\begin{document}
\maketitle

\begin{abstract}
In this article, we study the family of elliptic curves $E/\Q$, having
good reduction at $2$ and $3$, and whose $j$-invariants are
small. Within this set of elliptic curves, we consider the following
two subfamilies: first, the set of elliptic curves $E$ such that the
ratio $\Delta(E)/C(E)$ is squarefree; and second, the set of elliptic
curves $E$ such that $\Delta(E)/C(E)$ is bounded by a small power
$(<3/4)$ of $C(E)$. Both these families are conjectured to contain a
positive proportion of elliptic curves, when ordered by conductor.

Our main results determine asymptotics for both these families, when
ordered by conductor. Moreover, we prove that the average size of the
$2$-Selmer groups of elliptic curves in the first family, again when
these curves are ordered by their conductors, is $3$. This implies
that the average rank of these elliptic curves is finite, and bounded
by $1.5$.
\end{abstract}

\section{Introduction}
Every elliptic curve over $\Q$ can be uniquely represented as
$E_{AB}:y^2=x^3+Ax+B$, where $A$ and $B$ are integers such that there
is no prime $p$ with $p^4\mid A$ and $p^6\mid B$, and such that
$\Delta(A,B):=-4A^3-27B^2\neq 0$. Given an elliptic curve $E$ over
$\Q$, we denote its algebraic rank by $r(E)$ and its analytic rank by
$r_\an(E)$. The Birch and Swinnerton-Dyer conjecture asserts that
these two quantities are equal, i.e., we have $r(E)=r_\an(E)$.

Foundational conjectures of Goldfeld \cite{goldfeld-quadtwistconj} (in
the case of families of quadratic twists of elliptic curves) and
Katz--Sarnak \cite{katzsarnak} (for the full family of elliptic
curves) assert that a density of $50\%$ of elliptic curves have rank
$0$, and that $50\%$ have rank $1$, and that the average rank of
elliptic curves is $1/2$. Both these conjectures are formulated
through a study of the associated family of the $L$-functions $L_E(s)$
attached to the elliptic curves $E$. The behaviour of $L_E(s)$ at and
near the critical point is used to control the distribution of
analytic ranks, which, assuming the BSD conjecture, can be used to
give heuristics for the distribution of the algebraic ranks.

The most natural way to order a family of $L$-functions is by their
{\it conductors}, which, in this case of $L$-functions of elliptic
curves, is equal to the levels of the associated modular forms. Thus
in the conjectures of Goldfeld and Katz--Sarnak, it is implicitly
assumed that elliptic curves are ordered by their conductors.
However, when studying two-parameter families of elliptic curves, the
curves $E_{AB}$ are usually ordered by their (naive) {\it height}
$H(E_{AB})=\max\{4|A|^3,27B^2\}$.\footnote{See, however, work of
  Hortsch \cite{HortschFaltH} obtaining asymptotics for the number of
  elliptic curves with bounded Faltings height.} Assuming the
generalized Riemann hypothesis, Brumer \cite{brumer}, Heath-Brown
\cite{heathbrown-avgrank}, and Young \cite{young-avgrank}, proved the
successively better bounds of 2.3, 2, and $25/14$, on the average
analytic ranks of elliptic curves when ordered by height. On the
algebraic side, Bhargava and the second named author \cite{bs5sel}
proved that the average rank of elliptic curves, when ordered by
height, is bounded by $0.885$.

If elliptic curves are instead ordered by conductor, even asymptotics
for the number of curves are not known. The {\it discriminant}
$\Delta(E_{AB})$ of $E_{AB}$ is (up to absolutely bounded factors of
$2$ and $3$) $-4A^3-27B^2$. The {\it conductor} $C(E_{AB})$ of
$E_{AB}$ is (again, up to bounded factors of $2$ and $3$) the product
over all primes $p$ dividing $\Delta(E_{AB})$ of either $p$ or $p^2$
depending on if $E_{AB}$ has multiplicative or additive reduction at
$p$. Building on the work of Brumer--McGuinnes \cite{brumermcguinness}
on the family of elliptic curves ordered by discriminant, Watkins
\cite{watkins-heuristics} gives heuristics suggesting that the number
of elliptic curves with conductor bounded by $X$ grows as $\sim
cX^{5/6}$ for an explicit constant~$c$. Lower bounds of this magnitude
are easy to obtain, but the best known upper bound is
$O(X^{1+\epsilon})$ due to work of Duke--Kowalski \cite{DukeKowalski}.

The difficulties in determining precise upper bounds are
twofold. First, it is difficult to rule out the possibility of many
elliptic curves with large height but small discriminant. Second, it
is difficult to rule out the possibility of many elliptic curves with
large discriminant but small conductor. It is interesting to note here
that the second difficulty is exactly a nonarchimedean version of the
first. Indeed, curves $E_{AB}$ with large height and small
discriminant correspond to pairs $(A,B)$ of integers, where $4A^3$ and
$-27B^2$ are unusually close as real numbers. On the other hand,
curves $E_{AB}$ with large discriminant and small conductor correspond
to pairs of integers $(A,B)$ such that $4A^3$ and $-27B^2$ are
unusually close as $p$-adic numbers.

In this article, we focus on studying the second difficulty while
entirely sidestepping the first. To this end, we let $\E$ denote the
set of elliptic curves $E$ over $\Q$ that satisfy the following
properties.
\begin{enumerate}
\item The $j$-invariant $j(E)$ of $E$ satisfies $j(E)<\log\Delta(E)$.
\item $E$ has good reduction at $2$ and $3$.
\end{enumerate}
The first of the above three properties excludes all elliptic curves
$E$ with $\Delta(E)\ll H(E)^{1-\epsilon}$ and is absolutely critical
for our results. According to the Brumer--Mcguinnes heuristics
\cite{brumermcguinness}, only a negligible number of elliptic curves
are being excluded by the assumption of this property, but this is
unproven. The second property is a technical assumptions made to
simplify local computations at the $2$-adic and $3$-adic places. We
will in fact have to further restrict our families of elliptic
curves. Define the families
\begin{equation*}
\begin{array}{rcl}
\E_\sf&:=&
\displaystyle \Bigl\{E\in \E:\frac{\Delta(E)}{C(E)}\;
\mbox{is squarefree} \Bigr\},\\[.2in]
\E_\kappa&:=&
\displaystyle \bigl\{ E\in\E:\Delta(E)<C(E)^{\kappa}\bigr\},
\end{array}
\end{equation*}
for every $\kappa>1$. When ordered by conductor, the family
$\E_\kappa$ conjecturally contains $100\%$ of elliptic curves with
good reduction at $2$ and $3$, and $\E_\sf$ conjecturally contains a
positive proportion of elliptic curves. We prove the following result
determining asymptotics for these families of elliptic curves, ordered
by their conductors.
\begin{thm}\label{thmmain}
Let $1<\kappa<7/4$ be a positive constant. Then we have
\begin{equation}\label{eqEF}
\begin{array}{rcl}
\displaystyle\#\{E\in \E_\sf:\; C(E)<X\}&\sim&
\displaystyle \frac{1+\sqrt{3}}{60\sqrt{3}}
\frac{\Gamma(1/2)\Gamma(1/6)}{\Gamma(2/3)}
\cdot\prod_{p\geq 5}\Bigl(1+\frac{1}{p^{7/6}}-\frac{1}{p^2}-\frac{1}{p^{13/6}}\Bigr)
\cdot X^{5/6};\\[.2in]
\displaystyle\#\{E\in \E_\kappa:\; C(E)<X\}&\sim&
\displaystyle\frac{1+\sqrt{3}}{60\sqrt{3}}
\frac{\Gamma(1/2)\Gamma(1/6)}{\Gamma(2/3)}
\cdot\prod_{p\geq 5}
\Bigl[\Bigl(1-\frac{1}{p}\Bigr)\Bigl(1+\frac{1}{p^{5/3}}+\frac{1}{p^{11/6}}+\frac{1}{p^{17/6}}\Bigr)\\[.2in]
&&
+\displaystyle\frac{1}{p}\Bigl(1-\frac{1}{p}\Bigr)
\Bigl(1-\frac{1}{p^{1/6}}\Bigr)^{-1}\Bigl(1+\frac{2}{p}-\frac{2}{p^{3/2}}\Bigr)\Bigr]
\cdot X^{5/6}.
\end{array}
\end{equation}
\end{thm}
We expect Theorem \ref{thmmain} to hold for all $\kappa$. Furthermore,
since the abc conjecture implies that for $\kappa>6$, we have
$\E_\kappa=\E$, we expect these asymptotics to also hold for the
family $\E$. We note that the Euler factors appearing in Theorem
\ref{thmmain} arise naturally from the densities of elliptic curves
over $\Q_p$ with fixed Kodaira symbol. These densities are computed in
Theorem \ref{propcasessp}.

Our next main result is on the distribution of ranks of elliptic
curves in $\E_\sf$. As in \cite{bs2sel}, we study the ranks of these
elliptic curves via their $2$-Selmer groups. Recall that the
$2$-Selmer group $\Sel_2(E)$ of an elliptic curve $E$ over $\Q$ is a
finite $2$-torsion group which fits into the exact sequence
\begin{equation}\label{eqexseq}
0\to E(\Q)/2E(\Q)\to\Sel_2(E)\to\SH_E[2]\to 0,
\end{equation}
where $\SH_E$ denotes the Tate--Shafarevich group of $E$. Our result
regarding the $2$-Selmer groups of elliptic curves in $\E_\sf$ is as
follows.
\begin{thm}\label{thmsel}
When elliptic curves in $\E_\sf$ are ordered by their conductors, the
average size of their $2$-Selmer groups is $3$.
\end{thm}

Theorem \ref{thmsel} has the following immediate corollary.

\begin{cor}\label{corrank}
When elliptic curves in $E\in \E_\sf$ are ordered by their conductors,
their average $2$-Selmer rank is at most $1.5$; thus, their average
rank is at most $1.5$ and the average rank of $\SH_E[2]$ is also at most
$1.5$.
\end{cor}

Corollary \ref{corrank} provides evidence for the widely held belief
that the distribution of the ranks of elliptic curves are the same
regardless of whether the curves are ordered by height or
conductor. Moreover, as expected, the average size of the $2$-Selmer
groups of curves in $\E_\sf$ are the same as the average over all
elliptic curve ordered by height obtained in \cite[Theorem
  1.1]{bs2sel}.
We remark that our methods are flexible enough to recover verions of
Theorems \ref{thmmain} and \ref{thmsel} where the families $\E_\sf$
and $\E_\kappa$ are restricted so that the curves in them satisfy any
finite set of local conditions. This result is stated as Theorem
\ref{thmmainlarge}.

The key ingredient for proving the main results are ``uniformity
estimates'' or ``tail estimates''. These are upper bounds on the
number of elliptic curves in our families whose discriminants are large
compared to their conductors. For the proof of Theorem \ref{thmsel},
we additionally need bounds on the sum of the sizes of the $2$-Selmer
groups of elliptic curves in $\E_\sf$ with large discriminant and
small conductor. To this end, we prove the following result for the
family $\E_\sf$.

\begin{thm}\label{thunifsqi}
For positive real numbers $X$ and $M$, we have
\begin{equation*}
  \#\Bigr\{(E,\sigma):E\in\E_\sf,\;C(E)<X,\;\Delta(E)>MC(E),\;
  \sigma\in\Sel_2(E)\Bigl\}\,\ll_\epsilon
  \frac{X^{5/6+\epsilon}}{M^{1/6}}.
\end{equation*}
\end{thm}
We note that up to the power of $X^\epsilon$, this is expected to be
the optimal bound.

\vspace{.2in}
For the family $\E_\kappa$, we prove the following result.

\begin{thm}\label{thunifsmind}
Let $\kappa<7/4$ and $\delta>0$ be positive constants.
Then there exists a positive constant $\theta$, depending only on
$\delta$ and $\kappa$, such that for every $X>0$, we have
\begin{equation*}
  \#\Bigr\{E\in\E_\kappa:C(E)<X,\;\Delta(E)>X^\delta C(E)\Bigl\}\,\ll_\epsilon
    X^{5/6-\theta+\epsilon}.
\end{equation*}
\end{thm}

In \cite{bs2sel}, a version of such uniformity estimates were
proved. These estimates were used to obtain asymptotics on the number
of elliptic curves with bounded height and squarefree discriminant, as
well as to compute the average sizes of the $2$-Selmer groups of these
elliptic curves. One main input used in proving these estimates was
the Ekedahl sieve, as developed by Bhargava in
\cite{manjul-geosieve}. For our applications, this sieve falls short
of what is needed since our curves have much larger height than in the
previous case. Indeed, the height of $E\in\E_\sf$ with $C(E)=X$ can be
as large as $X^2$, in which case the Ekedahl sieve gives rise to an
error term of $O(X^{4/3})$ which is much too large. Subsequent
improvements to the Ekeshal sieve by Taniguchi--Thorne
\cite{TaniguchiThorneDistLev}, in which the sieve is combined with
equidistribution methods, are also insufficient for our purposes.

We now describe the proofs of our main theorems. We study the ratios
$\Delta(E)/C(E)$ of elliptic curves $E:y^2=f(x)$ in our families by
considering the associated family of cubic rings $R_f:=\Z[x]/f(x)$ and
cubic algebras $K_f:=\Q[x]/f(x)$ over $\Q$. Let $\cO_f$ denote the ring
of integers of $K_f$. Then $R_f$ is a suborder of $K_f$. Define the
invariants
\begin{equation*}
\begin{array}{rcl}
Q(E)&:=&[\cO_f:R_f]\\[.05in]
D(E)&:=&\Disc(K_f)
\end{array}
\end{equation*}
which satisfy the relation
\begin{equation*}
\Delta(E)=\Disc(R_f)=Q(E)^2D(E).
\end{equation*}
For primes $p$, we let $C_p(E)$, $\Delta_p(E)$, $Q_p(E)$, and $D_p(E)$
denote the $p$-parts of $C(E)$, $\Delta(E)$, $Q(E)$, and $D(E)$,
respectively.  The local invariants $C_p(E)$, $\Delta_p(E)$, $Q_p(E)$,
and $D_p(E)$ depend only on the Kodaira symbol of $E$. The starting
point of our proof is a determination of these local invariants along
with a computation of the density of elliptic curves over $\Q_p$ with
fixed Kodaira symbol.

\begin{thm}\label{propcasessp}
Fix a prime $p\geq 5$ and a Kodaira symbol $T$. Let $E:y^2=f(x)$ be an
elliptic curve over $\Z_p$ such that the Kodaira symbol of $E$ is
$T$. Then the local invariants of $E$ are as given in Table
\ref{tabloc}. Furthermore, there exists an element $t\in\Z_p$ such
that coefficients of $f(x+t)=x^3+ax^2+bx+c$ are as given in the second
column of Table \ref{tabloc}. Finally, the density of all elliptic
curves with Kodaira symbol $T$ is as given in the last column.

\begin{table}[ht]
\centering
\begin{tabular}{|c | c| c|c|c|c|c|}
\hline
Kodaira Symbol & Congruence Condition &
$C_p(E)$ & $\Delta_p(E)$ & $Q_p(E)$ & $D_p(E)$ & Density  \\
 of $E$ &&&&&&\\
\hline
$\I_0$     & $p\nmid\Delta(f)$ & $1$ & $1$ & $1$ & $1$ & $(p-1)/p$\\[.1in]
$\I_{n}$   & $p\nmid a, \quad p^{\lceil n/2\rceil}\mid b, \quad p^{n}\mid\mid c$ &
$p$ & $p^{n}$ & $p^{\lfloor n/2\rfloor}$ & $p^{n\!\!\!\pmod 2}$
&  $(p-1)^2/p^{n+2}$\\[.1in]
$\I\I$       & $p\mid a, \quad p\mid b,\quad p\parallel c$ &
$p^2$ & $p^2$ & $1$ & $p^2$ & $(p-1)/p^3$ \\[.1in]
$\I\I\I$       & $p\mid a, \quad p\parallel b,\quad p^2\mid c$ &
$p^2$ & $p^3$ & $p$ & $p$ & $(p-1)/p^4$ \\[.1in]
$\I\rV$       & $p\mid a, \quad p^2\mid b,\quad p^2\parallel c$ &
$p^2$ & $p^4$ & $p$ & $p^2$ & $(p-1)/p^5$ \\[.15in]

$\I_0^*$     & $p\mid a,\;p^2\mid b,\;p^3\mid c,\;p^7\nmid\Delta(f)$ &
$p^2$ & $p^6$ & $p^3$ & $1$ & $(p-1)/p^6$\\[.1in]
$\I_{n}^*$   &
$p\parallel a, \; p^{\lceil n/2\rceil+2}\mid b, \; p^{n+3}\parallel c$ &
$p^2$ & $p^{n+6}$ & $p^{\lfloor n/2\rfloor+3}$ & $p^{n\!\!\!\pmod 2}$
&  $(p-1)^2/p^{n+7}$\\[.1in]
$\I\rV^*$       & $p^2\mid a, \quad p^3\mid b,\quad p^4\parallel c$ &
$p^2$ & $p^8$ & $p^3$ & $p^2$ & $(p-1)/p^8$ \\[.1in]
$\I\I\I^*$       & $p^2\mid a, \quad p^3\parallel b,\quad p^5\mid c$ &
$p^2$ & $p^9$ & $p^4$ & $p$ & $(p-1)/p^9$ \\[.1in]
$\I\I^*$       & $p^2\mid a, \quad p^4\mid b,\quad p^5\parallel c$ &
$p^2$ & $p^{10}$ & $p^4$ & $p^2$ & $(p-1)/p^{10}$ \\
\hline
\end{tabular}
\caption{Local invariants of small elliptic curves}\label{tabloc}
\end{table}
\end{thm}

These density computations are straightforward, and indeed many of
them are implicit in the work of Watkins \cite[\S
  3.2]{watkins-heuristics}. However we include a proof since our use
of a $\G_a$-action on the space of monic cubic polynomials simplifies
the computations.

We use three different techniques to prove the estimates of Theorems
\ref{thunifsqi} and \ref{thunifsmind}. First, we fix a prime $p\geq 5$
and a Kodaira symbol $T$. The set of elliptic curves that have Kodaira
symbol $T$ at $p$ is cut out by certain congruence conditions $S$
modulo $q$, some power of $p$. Working modulo $q$, we compute the
Fourier transform of the characteristic function of $S$. An
application of Poisson summation then yields baseline estimates for
the number of elliptic curves with bounded height having Kodaira
symbol $T$ at $p$.

Our next two techniques average over primes $p$ in a crucial
way. Suppose that $E:y^2=f(x)$ is an elliptic curve in $\E_\sf$ such
that the ratio $\Delta(E)/C(E)$ is large. Then we prove that either
the discriminant of the algebra $K_f$ is small, or that the shape of
the ring of integers $\cO_f$ of $K_f$ is very skewed. The work of
Bhargava and Harron \cite{bhargavaharron} proves that the shapes of
rings of integers are equidistributed in the family of cubic
fields. Furthermore, the forthcoming thesis of Chiche-Lapierre
\cite{Lapierrethesis} determines asymptotics for the number of cubic
fields such that the shapes of their ring of integers are constrained
to lie within $0$-density sets. Using ideas from these works, we prove
bounds on the number of possible cubic algebras $K_f$ corresponding to
elliptic curves in $\E_\sf$ with bounded conductor, along with bounds
on the average sizes of the $2$-torsion subgroups $\Cl_2(K_f)$ of the
class groups of $K_f$. In combination with the work of Brumer--Kramer
\cite{brumerkramer}, relating the size of $\Sel_2(E)$ to
$\#\Cl_2(K_f)$, we deduce Theorem \ref{thunifsqi}.

The above method exploits the following crucial fact. If $E:y^2=f(x)$
is an elliptic curve such that the ratio $\Delta(E)/C(E)$ is large,
then primes $p$ such that the Kodaira symbol of $E$ at $p$ is $\I_0$,
$\I_1$, $\I_2$, $\I\I$, or $\I\I\I$ impose archimedean constraints on
the algebras $K_f$. However, primes $p$ with Kodaira symbol $\I\rV$ or
$\I_n$ with $n\geq 3$ impose only $p$-adic conditions on
$R_f\hookrightarrow K_f$. Namely, the prime $p$ divides the gcd of
$Q(E)$ and $D(E)$. To exploit this, we proceed as follows. The set of
integer monic traceless cubic polynomials $f$ with $p\mid Q(E_f)$
embeds into the space of binary quartic forms with a rational linear
factor. This embedding $\sigma$ is defined in \eqref{eqsig}. The group $\PGL_2$ acts on the space of binary quartic
forms, and the ring of invariants for this action is freely generated
by two polynomials $I$ and $J$. Restricted to the space of reducible
binary quartic forms gives an additional invariant $Q$. Explicitly, if
$g(x,y)$ is a binary quartic form with coefficients in $\Q$, and
$g(\alpha,\beta)=0$, then define
\begin{equation*}
Q(g(x,y),[\alpha:\beta])=\frac{g(x,y)}{\beta x-\alpha y}(\alpha,\beta).
\end{equation*}
This new invariant $Q$ is an exact analogue of the $Q$-invariants used
in \cite{BSWsf1},\cite{BSWsf2} to compute the density of polynomials with
squarefree polynomials. As there, for every fixed root
$[\alpha:\beta]\in\P^1(\Z)$, the discriminant polynomial on the space
of integer binary quartic forms $g$ with $g(\alpha,\beta)$ is
reducible, and in fact divisible by $Q^2$. We also define
\begin{equation*}
D(g(x,y),[\alpha:\beta]):=\Delta(g)/(Q(g(x,y),[\alpha:\beta]))^2.
\end{equation*}
Our embedding $\sigma$ satisfies $Q(E)=Q(\sigma(E))$ and
$D(E)=D(\sigma(E))$. Then the required estimates on elliptic curves
$E\in\E_\kappa$ with large $\Delta(E)/C(E)$, translate to estimates on
the number of $\PGL_2(\Z)$-orbits on integral reducible binary quartic
forms with bounded height and large $Q$- and $D$-invariants. We prove
the required estimates by fibering over roots, and then combining
geometry of numbers methods with the Ekedahl sieve.

This paper is organized as follows. In \S2 and \S3, we work locally,
one prime at a time. Theorem \ref{propcasessp} is proved in \S2, while
the Fourier coefficients corresponding to a fixed Kodaira symbol are
computed in \S3. The computation of the Fourier coefficients are then
used to obtain estimates (see Theorem \ref{thm:equimain}) on curves
with fixed Kodaira symbols at finitely many primes. We prove bounds on
the number of cubic fields $K$, weighted by $|\Cl_2(K)|$, in \S4, and
obtain estimates on the number of reducible integer binary quartic
forms with large $Q$- and $D$-invariants in \S5. The results of \S3,
\S4, and \S5, are combined in \S6 to prove the uniformity estimates
Theorems \ref{thunifsqi} and \ref{thunifsmind}. Finally, in \S7, we
prove the main results Theorems \ref{thmmain} and \ref{thmsel}.

\subsection*{Acknowledgments}
It is a pleasure to thank Manjul Bhargava, Benedict Gross, Hector
Pasten, Peter Sarnak, and Jacob Tsimerman for many helpful
conversations. The second named author is supported an NSERC Discovery
Grant and a Sloan Fellowship. The third named author is supported by
an NSERC Discovery Grant.

\section{Reduction types of elliptic curves}

Throughout this section, we fix a prime $p\geq 5$. Let $U$ denote the
space of monic cubic polynomials. Then for any ring $R$, we have
\begin{equation*}
U(R)=\{x^3+ax^2+bx+c:a,b,c\in R\}.
\end{equation*}
We denote the space of traceless elements of $U$ (i.e., $a=0$ in the
above equation) by $U_0$. The group $\G_a$ acts on $U$ via $(t\cdot
f)(x)=f(x+t)$. Given any element $f\in U(\Z_p)$, there exists a
unique element $\gamma\in\Z_p$ such that $f_0(x)=(\gamma\cdot f)(x)$
belongs to $U_0(\Z_p)$. Thus we may identify the quotient space $\Z_p\backslash
U(\Z_p)$ with $U_0(\Z_p)$. We denote the Euclidean measures on
$U(\Z_p)$ and $U_0(\Z_p)$ by $dg=da\,db\,dc$ and $df=db\,dc$,
respectively, where $da$, $db$ and $dc$ are Haar measures on $\Z_p$
normalized so that $\Z_p$ has volume $1$.
Then the change of measure formula for the bijection
\begin{equation}\label{eqzpaction}
\begin{array}{rcl}
\Z_p\times U_0(\Z_p)&\to& U(\Z_p) \\[.05in]
(t,f(x))&\mapsto& g(x)=(t\cdot f)(x)=f(x+t)
\end{array}
\end{equation}
is $dt\,df=dg$, where $dt$ is again the Haar measure on $\Z_p$ normalized
so that $\Z_p$ has volume $1$.

Given an element $f(x)\in U(\Z_p)$ such that the discriminant
$\Delta(f)$ is nonzero, we consider the elliptic curve $E_f$ over
$\Q_p$ with affine equation $y^2=f(x)$. An element $f(x)\in U(\Z_p)$
with nonzero discriminant is said to be {\it minimal} if
$\Delta(f)=\Delta(E_f)$. Equivalently, $f(x)$ is minimal if
$f_0(x)=x^3+Ax+B$, the unique element in $U_0(\Z_p)$ in the
$\Z_p$-orbit of $f$, does not satisfy $p^4\mid A$ and $p^6\mid
B$. Another equivalent condition is that the roots of $f_0(x)$ are not
all multiples of $p^2$. We denote the set of minimal elements in
$U(\Z_p)$ by $U(\Z_p)^\min$, and denote $U(\Z_p)^\min\cap U_0(\Z_p)$ by
$U_0(\Z_p)^\min$. The map $f\mapsto E_f$ is then a natural surjective
map from $\Z_p\backslash U(\Z_p)^\min$ (equivalently $U_0(\Z_p)^\min$)
to the set of isomorphism classes of elliptic curves over $\Q_p$.

The twisting-by-$p$ map is a natural involution on the set of
isomorphism classes of elliptic curves over $\Q_p$. This yields a
natural involution $\sigma$ on $\Z_p\backslash U(\Z_p)^\min$. If $f\in
U(\Z_p)^\min$ such that $f_0(x) = x^3+Ax+B$ with $p^2\nmid A$ or
$p^3\nmid B$, then we say $f$ is {\it small} and in this case,
$\sigma(f)_0(x) = \sigma(f_0)(x)=x^3 + p^2Ax + p^3B$. Otherwise, if $f_0(x) =
x^3+Ax+B$ with $p^2\mid A$ and $p^3\mid B$, then we say $f$ is {\it
  large} and in this case, $\sigma(f)_0(x)=\sigma(f_0)(x) = x^3 +
p^{-2}Ax + p^{-3}B$. We have $\Delta(E_{\sigma(f)}) = p^6\Delta(E_f)$
if $f$ is small and $\Delta(E_{\sigma(f)}) = p^{-6}\Delta(E_f)$
otherwise. Let $U(\Z_p)^\sm$ denote the set of small elements $f\in U(\Z_p)$.

Let $E$ be an elliptic curve over $\Q_p$, and let $\mathscr{X}$ be a
minimal proper regular model of $E$ over $\Z_p$. For brevity, we will
say that $T$, the Kodaira symbol associated to the special fiber of
$\mathscr{X}$, is the Kodaira symbol of $E$. Define the {\it index} of
$E$ by $\ind(E):=\Delta(E)/C(E)$. Then the index of $E$ is $1$ if and
only if the Kodaira symbol of $E$ is $\I_0$ (when $E$ has good
reduction), $\I_1$, or $\I\I$. Given $f\in U(\Z_p)^\min$, we define
the {\it index} of $f$ to be $\ind(f):=\ind(E_f)$. We also define two
other invariants associated to elements $f\in U(\Z_p)^\min$. Let $K_f$
denote the cubic etal\'e algebra $K_f:=\Q_p[x]/f(x)$, let $\cO_f$
denote the ring of integers of $K_f$, and let $R_f$ denote the cubic
ring $\Z[x]/f(x)$. We define
\begin{equation*}
\begin{array}{rcl}
Q_p(f)&:=&[\cO_f:R_f];\\[.05in]
D_p(f)&:=&\Disc(K_f).
\end{array}
\end{equation*}
These quantities are clearly invariant under the action of $\Z_p$ on
$U(\Z_p)$ and satisfy the equation
$$\Delta(f)=\Delta(R_f)=D_p(f)Q_p(f)^2.$$

The next result gives a criterion for $f\in U(\Z_p)$ to be small in
terms of the Kodaira symbol of $E_f$.

\begin{proposition}\label{propKSloc}
Let $f(x)\in U(\Z_p)^\min$ be a monic cubic polynomial corresponding
to the elliptic curve $E_f$. Then $f$ is small if and only if the
Kodaira symbol of $E_f$ is $\I_n$ for $n\geq 1$, $\I\I$, $\I\I\I$, or
$\I\rV$.
\end{proposition}
\begin{proof}
We first note that if the Kodaira symbol of $E_f$ is $\I\I$, $\I\I\I$,
or $\I\rV$, then $f$ is small since the discriminant is less than
$p^6$. If $f$ is not small, then $E_f$ has additive reduction. Hence
if the Kodaira symbol of $E_f$ is $\I_n$, then $f$ is small.

Conversely, we start with fixing an element $f\in U(\Z_p)^\sm$.  Let
$\alpha_1,\alpha_2,\alpha_3$ denote the three roots of $f(x)$ over
$\overline{\Q}_p$, the Galois closure of $\Q_p$, and let $\nu_p$
denote the $p$-adic valuation on $\overline{\Q}_p$.  We now consider
the following four cases: $f(x)$ is irreducible over $\Q_p$; $f(x)$
factors as a product of a linear and a quadratic factor over $\Q_p$,
and moreover $E_f$ has additive reduction; $f(x)$ factors as a product
of a linear and a quadratic polynomial over $\Q_p$, and moreover $E_f$
has multiplicative reduction; $f(x)$ factors into the product of three
distinct linear polynomials over $\Q_p$. In what follows, we will
repeatedly use \cite[Table 4.1]{silvermanEC2} to determine the Kodaira
symbol of $E_f$ from its reduction type and discriminant.

First suppose $f(x)$ is irreducible over $\Q_p$. The absolute Galois
group of $\Q_p$ acts transitively on $\alpha_1,\alpha_2,\alpha_3$. Let
$\sigma$ be an element sending $\alpha_1$ to $\alpha_2$. If
$\sigma(\alpha_3) = \alpha_3$, then $\sigma(\alpha_1 - \alpha_3) =
\alpha_2 - \alpha_3$. If $\sigma(\alpha_3) = \alpha_1$, then
$\sigma(\alpha_2 - \alpha_3) = \alpha_3 - \alpha_1$. Hence in either
case $$\nu_p(\alpha_1 - \alpha_3) = \nu_p(\alpha_2 - \alpha_3).$$
Similarly, we have $\nu_p(\alpha_1 - \alpha_3) = \nu_p(\alpha_1 -
\alpha_2)$. Let $m\in\frac13\Z$ be their common value. Let $t =
(\alpha_1 + \alpha_2 + \alpha_3)/3 \in\Z_p$. Then replacing $\alpha_i$
by $\alpha_i - t$, we may assume $\nu_p(\alpha_i)\geq m$ for $i =
1,2,3$. On the other hand, $\nu_p(\alpha_1 -
\alpha_2)\geq\max\{\nu_p(\alpha_1),\nu_p(\alpha_2)\}$. Hence
$\nu_p(\alpha_i) = m$ for $i = 1,2,3.$. Since $f$ is integal and small,
we have $0\leq m<1$.

If $m = 0$, then $E_f$ has good reduction at $p$, and the Kodaira
symbol is $\I_0$. If $m=1/3$, then $E_f$ has additive reduction at $p$
and $\nu_p(\Delta(E_f))=2$. This implies that the Kodaira symbol is
$\I\I$.  Finally, if $m = 2/3$, then $E_f$ has additive reduction and
$\nu_p(\Delta(E_f))=4$. It follows that the Kodaira symbol is $\I\rV$.

Next suppose that $f(x)$ factors as a product of a linear and a
quadratic factor over $\Q_p$, and that $E_f$ has additive reduction.
Let $\alpha_1$ denote the root of the linear factor and let $\alpha_2$
and $\alpha_3$ denote the conjugate roots of the quadratic
factor. Then $\alpha_1,\alpha_2,\alpha_3$ are congruent modulo
$p^{1/2}$. Let $t = (\alpha_1 + \alpha_2 + \alpha_3)/3$ as
above. Replacing $\alpha_i$ by $\alpha_i-t$, we may
assume $$\nu_p(\alpha_1)\geq 1,\quad\nu_p(\alpha_2) = \nu_p(\alpha_3) =
\frac12.$$ The latter equality holds because if $p\mid\alpha_2$, then
$f$ is not small. Since $\alpha_2$ and $\alpha_3$ are roots of a quadratic
polynomial $q(x)$ with $\Z_p$ coefficients, we have $\alpha_2 +
\alpha_3\in p\Z_p$. Hence
$$\nu_p(\alpha_2 - \alpha_3) = \frac12.$$ Clearly, $\nu_p(\alpha_1 -
\alpha_2) = \nu_p(\alpha_1 - \alpha_3) = 1/2$. Hence, $E_f$ has additive
reduction and $\nu_p(\Delta(E_f) = 3$. This implies that the Kodaira
symbol is $\I\I\I$.

The third case follows immediately: since $E_f$ is assumed to have
multiplicative reduction, the Kodaira symbol is $\I_n$ for some $n\geq
1$, which is sufficient. Finally, for the fourth case, suppose that
$f$ factors into a product of three linear polynomials over $\Q_p$. If
the $\alpha_i$ are all congruent modulo $p$, then replacing each
$\alpha_i$ by $\alpha_i-\alpha_1$, we see that that $f$ is not
small. Hence $E_f$ does not have additive reduction, and the Kodaira
symbol is again $\I_n$ for some $n\geq 0$. This conclude the proof of
the proposition.
\end{proof}



Next, we prove Theorem \ref{propcasessp}.

\vspace{.1in}

\noindent {\bf Proof of Theorem \ref{propcasessp}:} We start by
assuming that $E=E_f$ corresponds to $f\in U(\Z_p)^\sm$. By
Proposition~\ref{propKSloc}, the associated Kodaira symbol is $\I_n$,
$\I\I$, $\I\I\I$, or $\I\rV$. We will begin with verifying the second
through sixth columns of the Table \ref{tabloc}, leaving the density
computation to Proposition \ref{propKSden}. The result is clear if
$E_f$ has good reduction, which happens precisely when
$\Delta_p(E_f)=1$.

First assume that $E_f$ has additive reduction, in which case
$C_p(E)=p^2$. Then the Kodaira symbol of $E_f$ is $\I\I$, $\I\I\I$, or
$\I\rV$, according to whether $\Delta_p(E_f)$ is $p^2$, $p^3$, or $p^4$,
respectively. Replacing $f(x)$ with a $\Z_p$-translate, if necessary,
we may assume that $f(x)\equiv x^3\pmod{p}$. Write
$f(x)=x^3+pa_1x^2+pb_1x+pc_1$ with
$a_1,b_1,c_1\in\Z_p$. Then $$\Delta(f)\equiv
4p^3b_1^3-27p^2c_1^2+18p^3a_1b_1c_1\pmod{p^4}$$ and $p^2\parallel
\Delta(f)$ if and only if $p\nmid c_1$. In that case, the paragraph
following Lemma 13 in \cite{bstcubic} implies that $R_f$ is the
maximal order of $K_f$. This confirms the second through sixth columns
in the case when the Kodaira symbol is $\I\I$.

If $p^3\mid\Delta(f)$, then $p\mid c_1$. We write $c_1=pc_2$ for some
$c_2\in\Z_p$, and then $\Delta(f)\equiv 4p^3b_1^3\pmod{p^4}$. Hence
$p^3\parallel \Delta(f)$ if and only if $p\nmid b_1$. Suppose $p\nmid
b_1$. Then $R_f$ is a suborder of index $p$ in the cubic ring $\cO$
corresponding to the binary cubic form
$px^3+pa_1x^2y+b_1xy^2+c_2y^3$. The ring $\cO$ is maximal (from
\cite{bstcubic} as before) with $\Delta_p(\cO)=p$, confirming the
values of $Q_p$ and $D_p$ when the Kodaira symbol is
$\I\I\I$. Finally, suppose that $p\mid b_1$ and write $b_1=pb_2$. Then
$f(x)=x^3+pa_1x+p^2b_2+p^2c_2$, and since $f$ is small, we have
$p\nmid c_2$. In this case, we see that $\Delta_p(f)=p^4$ and that
$R_f$ is a suborder of index $p$ of the maximal order $\cO$
corresponding to the binary cubic form $px^3+pax^2y+pb_2xy^2+c_2y^3$
with $\Delta_p(\cO)=p^2$. This confirms the second through sixth
columns of Table \ref{tabloc} in the case when $E_f$ has additive
reduction and $f\in U(\Z_p)^\sm$.

Next, assume that $E_f$ has multiplicative reduction. From the proof
of Proposition \ref{propKSloc}, it follows that $f(x)$ is not
irreducible over $\Q_p$. Suppose that $f(x)$ factors into a product of
a quadratic $q(x)$ and a linear polynomial $\ell(x)$ over $\Q_p$ and
that $f(x)$ has splitting type $(1^21)$. Let $\alpha_1$ denote the
root of the linear factor and let $\alpha_2$ and $\alpha_3$ denote the
conjugate roots of the quadratic factor. Let $t = (\alpha_2 +
\alpha_3)/2\in\Z_p$. Replacing $\alpha_i$ by $\alpha_i-t$, we may
assume $$\nu_p(\alpha_1) = 0,\quad \nu_p(\alpha_2) = \lambda,\quad
\alpha_3 = -\alpha_2$$ for some positive
$\lambda\in\frac{1}{2}\Z$. Then $\nu_p(\alpha_2 - \alpha_3) =
\nu_p(2\alpha_2) = \lambda$. Clearly, $\nu_p(\alpha_1 - \alpha_2) =
\nu_p(\alpha_1 - \alpha_3) = 0$. Thus, $\Delta_p = p^{2\lambda}$ and
the Kodaira symbol of $E_f$ is $\I_{2\lambda}$. Moreover, the
coefficients $a$, $b$, and $c$ of $f$ satisfy
\begin{eqnarray*}
\nu_p(a) &=& \nu_p(\alpha_1 + \alpha_2 + \alpha_3) = \nu_p(\alpha_1) = 0,\\
\nu_p(b) &=& \nu_p(\alpha_1(\alpha_2 + \alpha_3) + \alpha_2\alpha_3) = 2\lambda,\\
\nu_p(c) &=& \nu_p(\alpha_1\alpha_2\alpha_3) = 2\lambda.
\end{eqnarray*}
The cubic order $\Z_p[x]/(f(x))$ is a suborder of index
$p^{\lfloor\lambda\rfloor}$ of the cubic order associated to the
binary cubic form
\begin{equation*}
p^{\lfloor\lambda\rfloor}x^3 + ax^2y +
(b/p^{\lfloor\lambda\rfloor})xy^2 + (c/p^{2\lfloor\lambda\rfloor})y^3,
\end{equation*}
which is maximal since its discriminant is $1$ when $\lambda$ is an
integer and $p$ when $\lambda$ is a half integer. Hence we have
$Q_p(E_f)= p^{\lfloor\lambda\rfloor}$ and $D_p = p^{2\lambda\pmod{2}}$ as
necessary.


Suppose instead that $f(x)$ factors as a product of three linear
polynomials over $\Q_p$. By assumption, the three roots $\alpha_1$,
$\alpha_2$, and $\alpha_3$ of $f(x)$ in $\Z_p$ are not all congruent
modulo $p$. After renaming, suppose $\alpha_2$ and $\alpha_3$ are
congruent modulo $p$ and $\alpha_1$ is not congruent to them. Let $t =
2\alpha_3-\alpha_2\in\Z_p$. Replacing $\alpha_i$ by $\alpha_i-t$, we
may assume $\alpha_1$ is a unit and $\alpha_2 = 2\alpha_3$. That
is, $$\nu_p(\alpha_1) = 0,\quad \nu_p(\alpha_2) = \nu_p(\alpha_3) =
\lambda,$$ for some positive integer $\lambda\in\Z$. Thus,
$\Delta_p(f) = p^{2\lambda}$, which implies that the Kodaira symbol of $E_f$
is $\I_{2\lambda}$. As a consequence, the coefficients $a$, $b$, and
$c$ of $f$ satisfy
\begin{eqnarray*}
\nu_p(a) &=& \nu_p(\alpha_1 + \alpha_2 + \alpha_3) = 0,\\
\nu_p(b) &=& \nu_p(\alpha_1(\alpha_2+\alpha_3) + \alpha_2\alpha_3) = \lambda,\\
\nu_p(c) &=& \nu_p(\alpha_1\alpha_2\alpha_3) = 2\lambda.
\end{eqnarray*}
The cubic order $\Z_p[x]/(f(x))$ is a suborder of index $p^{\lambda}$
of the cubic order associated to the binary cubic form $p^\lambda x^3
+ ax^2y + (b/p^\lambda)xy^2 + (c/p^{2\lambda})y^3$, which is maximal
since its discriminant is $1$. Therefore, $Q_p(E_f) = p^\lambda$ and
$D_p(E_f) = 1$ as required.

We now turn to large elliptic curves. Let $E$ be a large elliptic
curve over $\Z_p$. Let $E'$ denote the twist of $E$ by $p$. Then the
Kodaira symbol of $E'$ is $\I_n$, $\I\I$, $\I\I\I$, or $\I\rV$,
depending on whether the Kodaira symbol of $E$ is $\I_n^*$, $\I\rV^*$,
$\I\I\I^*$, or $\I\I^*$, respectively. Let $y^2=f(x)$ be a model for
$E'$, where the coefficients of $f(x)=x^3+ax^2+bx+c$ satisfy the
congruence conditions of Table \ref{tabloc}. Then
$y^2=g(x)=x^3+pax^2+p^2bx+p^3c$ is a model for $E$. It is then easy to
check that the second column of Table \ref{tabloc} is correct for all
ten rows. Furthermore, $K_g=K_f$ and $R_g$ has index $p^3$ in
$R_f$. It follows that the local invariants of $E$ are as in
Table~\ref{tabloc}. Theorem \ref{propcasessp} follows the density
computations in the following proposition. $\Box$

\begin{proposition}\label{propKSden}
The density of elliptic curves over $\Z_p$ having a fixed Kodaira
symbol is as in Table \ref{tabloc}.
\end{proposition}
\begin{proof}
Let $T$ be a fixed Kodaira symbol. Let $U(\Z_p)^{(T)}$
(resp.\ $U_0(\Z_p)^{(T)}$) denote the set of elements $f \in U(\Z_p)^\min$ (resp.\ $f \in U_0(\Z_p)^\min$) such that
$E_f$ has Kodaira symbol $T$. Then the density of elliptic curves with
Kodaira symbol $T$ is $\Vol(U_0(\Z_p)^{(T)})=\Vol(U(\Z_p)^{(T)})$,
where the equality holds since $\Z_p\cdot
U_0(\Z_p)^{(T)}=U(\Z_p)^{(T)}$ and the Jacobian change of variables of
the map \eqref{eqzpaction} is $1$.

We start with Kodaira symbol $\I_0$. The set $U(\Z_p)^{(\I_0)}$
consists of those $f\in U(\Z_p)$ such that $f(x)\pmod{p}$ has three
distinct roots in $\overline{\F}_p$. Denote these roots by $\alpha_1$,
$\alpha_2$, and $\alpha_3$. Either the $\alpha_i$ all belong to
$\F_p$, or $\alpha_1\in\F_p$ and $\alpha_2,\alpha_3$ are a pair of
conjugate elements in $\F_{p^2}\backslash\F_p$, or the $\alpha_i$ are
conjugate elements in $\F_{p^3}\backslash\F_p$. Thus, we have
\begin{equation*}
\begin{array}{rcl}
\displaystyle\Vol(U(\Z_p)^{(\I_0)})&=&
\displaystyle\frac{p(p-1)(p-2)}{6p^3}
+\frac{p(p^2-p)}{2p^3}+\frac{p^3-p}{3p^3}\\[.2in]
&=&\displaystyle 1-\frac{1}{p},
\end{array}
\end{equation*}
as required.

Second, we consider the Kodaira symbol $\I_n$ for $n\geq 1$. Suppose $f(x)\in U(\Z_p)^{\I_n}$. Then $f(x)$ has exactly
one double root modulo $p$. We therefore have $f(x) = g(x)(x-\alpha)$,
where $g(x)$ has a double root modulo $p$, and $p \nmid
g(\alpha)$. Clearly, we have $\Delta_p(g)=\Delta_p(f)=p^n$, since
$\Delta_p(E_f)=p^n$. We write the quadratic factor $g(x)$ in unique
form as $g(x)=(x+\beta)^2+\gamma$. The discriminant condition
translates to $p^n\parallel\gamma$, and the condition that $p\nmid
g(\alpha)$ translates to $p\nmid (\alpha+\beta)$.  Therefore, every
element of $U(\Z_p)^{\I_n}$ can be expressed uniquely in the form
\begin{equation*}
((x+\beta)^2+\gamma)(x-\alpha)=x^3+(2\beta-\alpha)x^2
+(\beta^2-2\alpha\beta+\gamma)x-\alpha\beta^2-\alpha\gamma,
\end{equation*}
such that $p^n\parallel\gamma$ and $p\nmid(\alpha+\beta)$. The
Jacobian change of variables for the map
$(\alpha,\beta,\gamma)\mapsto(a,b,c)$ is $-2(\alpha+\beta)^2-2\gamma$ which is
always a unit. Thus, we have
\begin{equation*}
\begin{array}{rcl}
\displaystyle\Vol(U(\Z_p)^{\I_n})&=&
\displaystyle\Vol(p^n\Z_p\backslash
p^{n+1}\Z_p)\Vol(\{(\alpha,\beta)\in\Z_p^2:p\nmid(\alpha+\beta)\})\\[.1in]
&=&\displaystyle (p-1)^2/p^{n+2},
\end{array}
\end{equation*}
as required.

Third, we consider the Kodaira symbols $\I\I$, $\I\I\I$, and
$\I\rV$. If $f\in U_0(\Z_p)$ is such that the Kodaira symbol of $E_f$
is one of the three above, then $f(x)=x^3+Ax+B$ has a triple root
modulo $p$, which implies that $p$ divides $A$ and $B$. By examining
the discriminant of $f$ as in the proof of Proposition
\ref{propKSloc}, we see that the Kodaira symbol of $E_f$ is $\I\I$ if
and only if $p\mid A$ and $p\parallel B$; $\I\I\I$ if and only if
$p\parallel A$ and $p^2\mid B$; and $\I\rV$ if and only if $p^2\mid A$
and $p^2\parallel B$. Hence the volumes of $U_0(\Z_p)^{(T)}$, for
$T=\I\I$, $\I\I\I$, and $\I\rV$, are $(p-1)/p^3$, $(p-1)/p^4$, and
$(p-1)/p^5$, as required.

Finally, we turn to the large Kodaira symbols, i.e., those
corresponding to large elliptic curves. Consider the following map
\begin{equation*}
\begin{array}{rcl}
\sigma: U(\Z_p)^\sm&\to& U(\Z_p)\\[.1in]
x^3+ax^2+bx+c&\mapsto& x^3+pax^2+p^2bx+p^3c.
\end{array}
\end{equation*}
Clearly, if $S\subset U(\Z_p)$ is any measurable set, then
$\Vol(\sigma(S))=p^{-6}\Vol(S)$. Furthermore, we set $\sigma(\I_n) =
\I_n^*$, $\sigma(\I\I) = \I\rV^*$, $\sigma(\I\I\I) = \I\I\I^*$ and
$\sigma(\I\rV)=\I\I^*$. Then $\sigma$ sends $f$ of Kodaira symbol $T$
to $\sigma(f)$ of Kodaira symbol $\sigma(T)$. Moreover, we have
$\sigma(t\cdot f)=(pt)\cdot\sigma(f)$. Hence we have
$$\sigma\bigl(U(\Z_p)^{(T)}\bigr)=\sigma\Bigl(\Z_p\cdot
U_0(\Z_p)^{(T)}\Bigr)
=(p\Z_p)\cdot\sigma\bigl(U_0(\Z_p)^{(T)}\bigr).$$ Fix any $g\in
U_0(\Z_p)^{(\sigma(T))}$. There exists $t\in\Z_p$ such that the
coefficients of $t\cdot g$ are as in the second column of Table
\ref{tabloc}. Hence there exists $f\in U(\Z_p)^{(T)}$ with
$\sigma(f)=t\cdot g$. Then $\sigma(f_0)$ is $\Z_p$-equivalent to
$g$. Since $\sigma(f_0)$ and $g$ both belong to $U_0(\Z_p)$, we must
have $\sigma(f_0)=g$. Hence we have
$\sigma\bigl(U_0(\Z_p)^{(T)}\bigr)=U_0(\Z_p)^{(\sigma(T))}$. Therefore,
we have

\begin{equation*}
\begin{array}{rcl}
\displaystyle\Vol\Bigl(U(\Z_p)^{(\sigma(T))}\Bigr)&=&
\displaystyle\Vol\Bigl(\Z_p\cdot U_0(\Z_p)^{(\sigma(T))}\Bigr)
\\[.2in]&=&
\displaystyle p\cdot \Vol\Bigl(p\Z_p\cdot U_0(\Z_p)^{(\sigma(T))}\Bigr)
\\[.2in]&=&
p\cdot \Vol\Bigl(\sigma\bigl(U(\Z_p)^{(T)}\bigr)\Bigr)
\\[.2in]&=&
p^{-5}\Vol\bigl(U(\Z_p)^{(T)}\bigr).
\end{array}
\end{equation*}
This concludes the proof of Proposition \ref{propKSden}, and thus of
Theorem \ref{propcasessp}.
\end{proof}

Theorem \ref{propcasessp} has the following immediate corollary, which
will be useful in what follows.
\begin{corollary}
Let $p\geq5$ be a prime. The density of elliptic curves $E$ over $\Q_p$ with good,
multiplicative, or additive reduction, such that
$\ind(E)=\Delta_p(E)/C_p(E)=p^k$ is as given in Table \ref{tabden}.

\begin{table}[ht]
\centering
\begin{tabular}{|c | c| c|c|c|}
\hline
Index & Good Red. & Multiplicative Red. & Additive Red. & Total\\
$1$   & $(p-1)/p$ & $(p-1)^2/p^3$ & $(p-1)/p^3$ & $(p^2-1)/p^2$ \\[.05in]
$p$   & 0 & $(p-1)^2/p^4$& $(p-1)/p^4$ & $(p-1)/p^3$ \\[.05in]
$p^2$ & 0 & $(p-1)^2/p^5$& $(p-1)/p^5$ & $(p-1)/p^4$ \\[.05in]
$p^3$ & 0 & $(p-1)^2/p^6$& $0$ &$(p-1)^2/p^6$\\[.05in]
$p^4$ & 0 & $(p-1)^2/p^7$& $(p-1)/p^6$ & $(2p-1)(p-1)/p^7$ \\[.05in]
$p^k$, $k=6,7,8$ &0 & $(p-1)^2/p^{k+3}$ &$(2p-1)(p-1)/p^{k+3}$ &
$(3p-2)(p-1)/p^{k+3}$\\[.05in]
$p^k$, $k=5$ or $k\geq 9$ &0 & $(p-1)^2/p^{k+3}$ &$(p-1)^2/p^{k+3}$&
$2(p-1)^2/p^{k+3}$\\
\hline
\end{tabular}
\caption{$p$-adic densities of elliptic curves with given index}\label{tabden}
\end{table}
\end{corollary}

%

\section{Fourier coefficients of polynomials with fixed Kodaira symbol}\label{sec:equi}

Let $p\geq 5$ be a prime, and let $U(\Z_p)^\min$ and $U(\Z_p)^\sm$ be
as in \S2.  Recall that to each $f(x)\in U(\Z_p)^\min$, we associate
the Kodaira symbol of the elliptic curve $E_f$. By Proposition
\ref{propKSloc} and Theorem \ref{propcasessp}, an element $f(x)\in
U(\Z_p)^\min$ belongs to $U(\Z_p)^\sm$ and satisfies $\Delta(f) \neq
C(f)$ if and only if the Kodaira symbol of $f$ is $\I\I\I$, $\I \rV$,
or $\I_n$ for $n \geq 2$. Denote the set of polynomials $f(x)\in
U(\Z)$ such that $f\in U(\Z_p)^\min$ for all primes $p$ by
$U(\Z)^\min$. Given $f(x)\in U(\Z)^\min$ and a prime $p$, we say that
the {\it Kodaira symbol of $f$ at $p$} is $T$, the Kodaira symbol of
$f(x)$ considered as an element in $U(\Z_p)^\min$.


Let $\Sigma$ be a set consisting of the following data: a finite set
$\{p_1,\ldots,p_k\}$ of primes $p_i\geq 5$ along with a Kodaira symbol
$T(p_i)$ which is $\I\I\I$, $\I\rV$ or $\I_{n\geq 2}$ associated to
each prime $p_i$ in the set. We say $f\in U(\Z)$ has splitting type $\Sigma$ if $f$ has Kodaira symbol $T(p_i)$ at
each prime $p_i$ in $\Sigma$. Let $U(\Z)_\Sigma$ denote the set of
elements $f\in U(\Z)$ with splitting type $\Sigma$. Given such a collection $\Sigma$, we
define the constant $Q(\Sigma)$ to be $\prod_{p_i}p_i^{a_i}$, where
$a_i=1$ if $T(p_i)$ is $\I\I\I$ or $\I\rV$, and $a_i=\lfloor n/2
\rfloor$ if $T(p_i)$ is $\I_n$. Note that if $f\in U(\Z)_\Sigma$, then
$Q(\Sigma)\mid Q(f)$. We define $m_{T}(\Sigma)$ to be the product of
all primes $p$ such that $T(p) = T$. We also define $m_\odd(\Sigma)$ to
be the product of all primes $p$ in $\Sigma$ such that
$\sigma(p)=\I_n$ for some \emph{odd} integer $n$. Finally, we define
$\nu(\Sigma)$ to be the product over the primes $p$ in $\Sigma$ of the
density $\nu(T_p)$, i.e., the $p$-adic volume of the set of elements
in $U(\Z_p)^{\min}$ having Kodaira symbol $T(p)$.

Define the height function $H$ on $U(\R)$ to be
\begin{equation*}
H(x^3+ax^2+bx+c):= \max\{|a|^6,|b|^3,|c|^2\}.
\end{equation*}
The goal of this section is to obtain a bound on the number of
elements in $U(\Z)$ that have bounded height and specified Kodaira
symbols $\I\I\I$, $\I\rV$ or $\I_{n\geq 2}$ at finitely many primes.
We prove the following theorem.
\begin{theorem}\label{thm:equimain}
  Let $\Sigma$ be as above and for every Kodaira symbol $T$, denote
  $Q(\Sigma)$, $m_\odd(\Sigma)$, and $m_T(\Sigma)$ by $Q$, $m_\odd$,
  and $m_T$, respectively. Then we have
  \begin{equation*}
\#\{f\in U(\Z)_\Sigma:H(f)<Y\}\ll_\epsilon
\frac{Y}{Q^2m_{\I\I\I}m_{\I\rV}^2m_\odd}+\frac{Qm_\odd}{m_{\I\rV}} Y^\epsilon,
  \end{equation*}
  where the implied constant is independent of $Y$ and $\Sigma$.
\end{theorem}

This section is organized as follows. First, in \S3.1, we recall some
preliminary results from Fourier analysis. In particular, the
``twisted Poisson summation'' formula of Proposition \ref{propTPS}
will be our main tool in proving Theorem \ref{thm:equimain}. Also, in
\eqref{eqFDtrans}, we determine how the action of $\G_a$ on $U$
changes the Fourier coefficients of functions. Next, in \S3.2, we
compute the Fourier coefficients of a slighly modified version of the
characteristic functions of the set of monic polynomials having
Kodaira symbol $T$, for $T=\I\I\I$, $\I\rV$, and $\I_{n\geq
  2}$. Finally, in \S3.3, we use these computations and the twisted
Poisson summation formula to prove Theorem \ref{thm:equimain}.

\subsection{Preliminary results from Fourier analysis}

We fix a positive integer $N$ with $(N,6)=1$, and consider the space
$\widehat{U(\Z/N\Z)}$ dual to $U(\Z/N\Z)$. We write elements $\chi\in
\widehat{U(\Z/N\Z)}$ as triples
$\chi=(\check{a},\check{b},\check{c})\in(\Z/N\Z)^3$, and view $\chi$
as the character given by
\begin{equation}\label{eqchiaction}
  \chi(x^3+ax^2+bx+c)=
  e\Bigl(\frac{\check{a}\cdot a+\check{b}\cdot b+\check{c}\cdot
  c}{N}\Bigr),
\end{equation}
where $e(x) := \exp(2\pi ix)$.
Given a function $\phi:U(\Z/N\Z)\to\C$, we have the Fourier dual
$\hat{\phi}:\widehat{U(\Z/N\Z)}\to\C$ defined to be
\begin{equation*}
\hat\phi(\chi):=\sum_{f\in U(\Z/N\Z)}\phi(f)\chi(f),
\end{equation*}
and Fourier inversion yields the equality
\begin{equation*}
\frac{1}{N^3}\sum_\chi\hat{\phi}(\chi)\overline{\chi(f)}=\phi(f).
\end{equation*}

The additive group $\Z/N\Z$ acts on the space $U(\Z/N\Z)$ via the
action $(r\cdot f)(x)=f(x+r)$. Identifying $U(\Z/N\Z)$ with the
coefficient space $(\Z/N\Z)^3$, we write the action explicitly:
\begin{equation*}
r\cdot (a,b,c)=((a+3r),(b+2ra+3r^2),(c+rb+r^2a+r^3)).
\end{equation*}
Given a function $\phi:U(\Z/N\Z)\to\C$ and an element $r\in\Z/N\Z$, we
define $r\cdot\phi:U(\Z/N\Z)\to\C$ to be
$(r\cdot\phi)(f):=\phi((-r)\cdot f)$. We also define an action of
$\Z/N\Z$ on $\widehat{U(\Z/N\Z)}$ by
\begin{equation}\label{eq:action}
  r\cdot\chi:=
  ((\check{a}+2r\check{b}+r^2\check{c}),(\check{b}+r\check{c}),\check{c}),
\end{equation}
for $\chi=(\ca,\cb,\cc)$.
Then we have
\begin{equation}\label{eqFDtrans}
\begin{array}{rcl}
  \displaystyle\widehat{r\cdot\phi}(\chi)&=&
  \displaystyle\sum_f (r\cdot\phi)(f)\chi(f)
  =\displaystyle\sum_f \phi((-r)\cdot f)\chi(f)
  =\displaystyle\sum_f \phi(f)\chi(r\cdot f)\\[.2in]
  &=&\displaystyle\sum_{f=(a,b,c)}
  \phi(f)e\Bigl(\frac{\check{a}a+\check{b}(b+2ra)+\check{c}(c+rb+r^2a)}{N}\Bigr)
  e\Bigl(\frac{3\check{a}r+3\check{b}r^2+\check{c}r^3}{N}\Bigr)\\[.2in]
  &=&\displaystyle e\Bigl(\frac{3\check{a}r+3\check{b}r^2+\check{c}r^3}{N}\Bigr)
  \sum_{f=(a,b,c)} \phi(f)e\Bigl(\frac{(\check{a}+2r\check{b}+r^2\check{c})a+
    (\check{b}+r\check{c})b+\check{c}c}{N}\Bigr)\\[.2in]
  &=&\displaystyle \Psi_r(\chi)\hat{\phi}(r\cdot\chi),
\end{array}
\end{equation}
where we set
\begin{equation*}
\Psi_r(\chi):=e\Bigl(\frac{3\check{a}r+3\check{b}r^2+\check{c}r^3}{N}\Bigr).
\end{equation*}

Note that if we identify elements
$\chi=(\check{a},\check{b},\check{c})\in\widehat{U(\Z/N\Z)}$ with
binary quadratic
forms $$P_\chi(x,y):=\check{a}x^2+2\check{b}xy+\check{c}y^2,$$ then
the action of $\Z/N\Z$ on $\widehat{U(\Z/N\Z)}$ in \eqref{eq:action}
corresponds exactly to the natural action: $$P_{r\cdot
  \chi}(x,y)=P_\chi(x,y+rx).$$ We define
$\Delta_2(\chi)=\check{b}^2-\check{a}\check{c}$. Then $\Delta_2$ is
invariant under the action of $\Z/N\Z$. Throughout the rest of \S4, we
will thus identify the space $\widehat{U(\Z/N\Z)}$ with the space
$V_2(\Z/N\Z)$, where $V_2=\Sym_2(2)$ is the space of binary quadratic
forms with middle coefficient a multiple of $2$.

Finally, we recall the following result which follows from the use of Poisson
summation combined with the unfolding technique.
\begin{proposition}\label{propTPS}
Let $\psi:U(\R)\to\R$ denote a smooth function with bounded
support. Let $\phi:U(\Z/N\Z)\to\R$ be any function. Then, for every
positive real number $Y$, we have
\begin{equation*}\label{fourier}
  \sum_{(a,b,c)\in U(\Z)}\psi\Bigl(
  \frac{a}{Y^{1/6}},\frac{b}{Y^{1/3}},\frac{c}{Y^{1/2}}\Bigr)\phi(a,b,c)=
  \frac{Y}{N^3}\sum_{\chi=(\check{a},\check{b},\check{c})\in \widehat{U(\Z)}}\hat{\psi}
  \Bigl(\frac{Y^{1/6}\check{a}}{N},\frac{Y^{1/3}\check{b}}{N},
  \frac{Y^{1/2}\check{c}}{N}\Bigr)
  \hat\phi(\check{a},\check{b},\check{c}).
\end{equation*}
The $\hat\psi$ on the right hand side is the usual Fourier transform
over $\R$ and so decays faster than any polynomial.
\end{proposition}

\subsection{Bounds on Fourier coefficients}

Let $p\geq 5$ be a fixed prime.  The conditions imposed by the choice
of Kodaira symbol $T$ being equal to $\I\I\I$, $\I\rV$ or $\I_{n\geq
  2}$ are defined via congruence conditions modulo $N=N_p(T)$, where $N$
is $p^2$, $p^2$ or $p^n$, respectively. Hence, when we refer to an
element $f$ having one of the above Kodaira symbols, we will be
implicitly assuming that $f$ belongs to $U(\Z/N\Z)$, where $N$ is the
appropriate power of $p$. Naturally, in this context, we will also
assume that elements $\chi$ belong to $\widehat{U(\Z/N\Z)}$, and
represent them as triplets $(\check{a},\check{b},\check{c})\in(\Z/N\Z)^3$.


For a Kodaira symbol $T\in\{\I\I\I,\I\rV,\I_{\geq 2}\}$, we
define the set $\CS_0(T)$ to be
\begin{equation*}
  \begin{array}{rcl}
    \{x^3+ax^2+bx+c:p\mid a;p\mid  b,p^2\mid c\}\subset U(\Z/p^2\Z)&\mbox{if}&T=\I\I\I;\\[.1in]
    \{x^3+ax^2+bx+c:p\mid a;p^2\mid b,p^2 \mid c\}\subset U(\Z/p^2\Z)&\mbox{if}&T=\I\rV;\\[.1in]
    \{x^3+ax^2+bx+c:p^{n}\mid b,p^{2n}\mid c\}\subset U(\Z/p^{2n}\Z)&\mbox{if}&T=\I_{2n};\\[.1in]
  \{x^3+ax^2+bx+c:p^{n+1}\mid b,p^{2n+1}\mid c\}\subset U(\Z/p^{2n+1}\Z)&\mbox{if}&T = \I_{2n+1}.\\
  \end{array}
\end{equation*}
From the second column of Table \ref{tabloc}, it follows that every
element having Kodaira symbol $T$ is contained within some $\G_a$
translate of $\CS_0(T)$. Let $\Phi_{0,T}$ denote the
characteristic function of $\CS_0(T)$, and define the function
$\Phi_{T}$ by
\begin{equation*}
\Phi_{T} = \sum_{r\in\Z/M\Z}
r\cdot\Phi_{0,T},
\end{equation*}
where $M=M_p(T)$ is $p$ if $T = \I\I\I, \I\rV$, $p^n$ if $T = \I_{2n}$
and $p^{n+1}$ if $T = \I_{2n +1}$.  The next
lemma, determining the Fourier transforms of the sets
$\Phi_{0,T}$, follows quickly from the definitions.
\begin{lemma}\label{lemequi1}
Let $p\geq 5$ be a prime number. Let $T$ be one the three Kodaira
symbols, and let $N=N_p(T)$ denote the appropriate power of $p$. For
$\chi=(\check{a},\check{b},\check{c})\in \widehat{U(\Z/N\Z)}$, we have
\begin{equation*}
\begin{array}{rclrcll}
\displaystyle|\widehat{\Phi_{0,\I\I\I}}(\chi)|&=&\left\{
\begin{array}{rcl}
p^2 &\mbox{ if }& p\mid\check{a},\,p\mid\check{b};\\
0 &\mbox{ else; }&
\end{array}
\right.
&\quad&
\displaystyle|\widehat{\Phi_{0,\I\rV}}(\chi)|&=&\left\{
\begin{array}{rcl}
p &\mbox{ if }& p\mid\check{a};\\
0 &\mbox{ else; }&
\end{array}
\right.\\[.3in]
\displaystyle|\widehat{\Phi_{0,\I_{2n}}}(\chi)|&=&\left\{
\begin{array}{rcl}
p^{3n} &\mbox{if}& p^{2n}\mid\check{a},\,p^n\mid\check{b};\\
0 &\mbox{else;}&
\end{array}
\right.
&\quad&
\displaystyle|\widehat{\Phi_{0,\I_{2n+1}}}(\chi)|&=&\left\{
\begin{array}{rcl}
p^{3n+1} &\mbox{if}& p^{2n+1}\mid\check{a},\,p^n\mid\check{b};\\
0 &\mbox{else.}&
\end{array}
\right.
\end{array}
\end{equation*}
\end{lemma}

As an immediate consequence, \ref{eqFDtrans} yield the
inequality
\begin{equation}\label{eqtranslatebound}
  |\widehat{\Phi_T}(\chi)|\leq p^{k_T}r_T(\chi),
\end{equation}
where $k_T$ is $2$,
$1$, $3n$, or $3n+1$ depending on whether $T$ is $\I\I\I$,
$\I\rV$, $\I_{2n}$, or $I_{2n+1}$, respectively, and $r_T(\chi)$ is
the number of $r\in \{0,\ldots,M-1\}$ such that $(r\cdot\chi)$ belongs
to the support of $\widehat{\Phi_{0,T}}$.  To bound $\widehat{\Phi_T}(\chi)$, it then remains to bound $r_T(\chi)$.  

\begin{lemma}\label{lemequi2}
We have
\begin{enumerate}
\item Let $T=\I\I\I$. Then $r_T(\chi)=0$ unless
  $p\mid\Delta_2(\chi)$. In that case, $r_T(\chi)=1$ if
  $p\nmid\chi$ and $r_T(\chi)=p$ otherwise.
\item Let $T=\I\rV$. Then $r_T(\chi)\leq 2$ if $p\nmid\chi$ and
  $r_T(\chi)=p$ otherwise.
\item Let $T=\I_{2n}$. Then $r_T(\chi)=0$ unless $\chi$ is $\G_a$-equivalent
  to some element $(0,p^{n+i}\check{b},p^j\check{c})$, for $i$ and $j$
  nonnegative integers and $p\nmid\cb\cc$. Then $r_T(\chi)\ll p^{\min(i,\lfloor j/2\rfloor)}$.
\item Let $T=\I_{2n+1}$. Then $r_T(\chi)=0$ unless $\chi$ is
  $\G_a$-equivalent to some element $(0,p^{n+i}\check{b},p^j\check{c})$, for
  $i$ and $j$ nonnegative integers and $p\nmid\cb\cc$. Then $r_T(\chi)\ll p^{\min(i,\lceil
    j/2\rceil)}$.
\end{enumerate}
\end{lemma}
\begin{proof}
We prove the above lemma in the case when $T=\I_{2n}$. Assume that
$\chi$ is $\G_a$-equivalent to $(0,p^{n+i}\cb,p^j\cc)$, for $i$ and $j$
  nonnegative integers and $p\nmid\cb\cc$. Note that the entry $p^j\cc$ does
not change under the $\G_a$-action. Then by definition, we have
\begin{equation*}
  r_T(\chi)=\#\bigl\{r\in\Z/p^n\Z:p^n\mid rp^j,\,
  p^{2n}\mid 2p^{n+i}r\cb+r^2p^j\cc\bigr\}.
\end{equation*}
Write $r\in\Z/p^n\Z$ as $r=sp^k+p^n\Z$ with $p\nmid s$. Then the condition on $r$ translates
to
\begin{equation*}
p^n\mid p^{j+k},\quad\quad p^{2n}\mid (2\cb p^{n+i+k}+s\cc p^{j+2k}).
\end{equation*}
We consider two posible cases. First assume that $p^{2n}$ divides both
$2\cb p^{n+i+k}$ and $s\cc p^{j+2k}$. Then we have $k\geq\max(n-i,
n-\lfloor j/2\rfloor)$, which implies that there are
$p^{\min(i,\lfloor j/2\rfloor)}$ choices for $r$. Otherwise, we have
$n+i+k=j+2k=:\ell<2n$, and $p^{2n-\ell}\mid 2\cb+s\cc$. In this case,
$s$ is determined modulo $p^{2n-\ell}$, which implies that there are
$p^{n-k-(2n-\ell)}=p^{\ell-n-k}$ choices for $r$. Note that
$\ell-n-k=i$. Furthermore, we have $j+2k=\ell<2n$, from which it
follows that $2(\ell-n-k)=2j+2k-2n<j$. This proves the lemma in the
case when $T=\I_{2n}$. The other three cases are similar, and we omit
the proof.
\end{proof}

\subsection{Proof of Theorem \ref{thm:equimain}}

Let $\Sigma$ be as before. That is, a finite set of primes $p\geq 5$,
and a Kodaira symbol $T_p=\I\I\I$, $\I\rV$, or $\I_{n\geq 2}$ for each
prime $p$ in this set. For each prime $p$ of $\Sigma$, set
$N_p:=N_p(T_p)$ and set $m_{\odd,p}$ to be $p$ if $T_p=\I_{2n+1}$ and
$1$ otherwise. Set $Q_p$ to be the $Q$-invariant associated to $T_p$
in Table \ref{tabloc}.  We define the quantities $N=N(\Sigma)$,
$Q=Q(\Sigma)$, and $m_\odd=m_\odd(\Sigma)$ to be the product over all
primes $p$ of $\Sigma$ of $N_p$, $Q_p$, and $m_p$, respectively. Note that $N = Q^2m_{\odd}$. Since $Q^2$ divides $\Delta(f)$ for any $f$ with splitting type $\Sigma$, we may assume that $Q$ and also $N$ are bounded above by some fixed power of $Y$.


Then elements with splitting type $\Sigma$ are defined via congruence
conditions modulo $N$. Let $\phi:U(\Z/N\Z)\to\R$ denote the
characteristic function of elements with splitting type $\Sigma$.  Let
$\psi:U(\R)\to\R_{\geq 0}$ be a smooth compactly supported function
such that $\psi(f)=1$ for $H(f)\leq 1$. We have

\begin{equation}\label{eqequibase}
\begin{array}{rcl}
\displaystyle\#\{f\in U(\Z)_\Sigma:H(f)<Y\}&\leq&
\displaystyle\sum_{(a,b,c)\in U(\Z)_\Sigma}\psi
\Bigl(\frac{a}{Y^{1/6}},\frac{b}{Y^{1/3}},\frac{c}{Y^{1/2}}\Bigr)
\phi(a,b,c)\\[.2in] &=&
\displaystyle\frac{Y}{N^3}
\sum_{\chi=(\check{a},\check{b},\check{c})\in \widehat{U(\Z)}}\hat{\psi}
\Bigl(\frac{Y^{1/6}\ca}{N},\frac{Y^{1/3}\cb}{N},\frac{Y^{1/2}\cc}{N}\Bigr)
\hat\phi(\ca,\cb,\cc)\\[.2in] &=&
\displaystyle S_0+S_{\ca\cc=0}+S_{\Delta_2=0}+S_{\neq 0},
\end{array}
\end{equation}
where $S_0$ is the contribution of the term $\chi=0$, $S_{\ca\cc=0}$ is
the contribution from the nonzero terms $\chi$ with $\ca\cc=0$,
$S_{\Delta_2=0}$ is the contribution from nonzero terms $\chi$ with
$\Delta_2(\chi)=0$, and $S_{\neq 0}$ is the contribution from
the terms $\chi$ with $\ca\cc\Delta_2(\chi)\neq 0$.  We bound each of
these quantities in turn.

To begin with, since $\hat{\phi}(0)/N^3=\nu(\Sigma)$ and $\psi$ is
compactly supported, we have
\begin{equation}\label{eqequis0}
S_0=\frac{Y}{N^3}\hat{\psi}(0)\hat{\phi}(0)\ll \nu(\Sigma)Y \ll \frac{Y}{Q^2m_{\I\I\I}m_{\I\rV}^2m_\odd},
\end{equation}
by Table \ref{tabloc}. To bound $S_{\ca\cc=0},$ $S_{\Delta_2=0}$ and $S_{\neq 0}$, we have the following immediate consequence of \eqref{eqtranslatebound} and Lemma \ref{lemequi2}.

\begin{corollary}
With notations as above, let $\chi = (\ca,\cb,\cc)\in \widehat{U(\Z)}$ with $\hat{\phi}(\chi)\neq 0$. Let $A$ be the largest divisor of $m_{\I\I\I}m_{\I\rV}$ dividing $\ca$, $\cb$ and $\cc$. For each prime $p$ with $T_p = I_{2n}$ or $T_p = I_{2n+1}$ for some $n\geq1$, let $k_p$ be the nonnegative integer with $p^{2n+k_p}\mid\mid \Delta_2(\chi)$. Then
\begin{equation}\label{eq:mj1}
\hat{\phi}(\chi) \ll A \, m_{\I\I\I}^2 \, m_{\I\rV} \prod_{T_p = I_{2n}} p^{3n+k_p/2} \prod_{T_p = I_{2n+1}} p^{3n+1+k_p/2}.
\end{equation}
\end{corollary}

Since $\hat\psi$ decays faster than any polynomial, it suffices to consider characters $\chi = (\ca,\cb,\cc)$ such that
$$\ca\ll N^{1+\epsilon}/Y^{1/6}, \quad \cb\ll N^{1+\epsilon}/Y^{1/3}, \quad \cc\ll N^{1+\epsilon}/Y^{1/2}.$$
We consider $S_{\ca\cc=0}$ first. Fix a divisor $A$ of $m_{\I\I\I}m_{\I\rV}$ and a nonnegative integer $k_p$ for every prime $p$ with $T_p = I_{\geq2}$. The number of characters $\chi = (0,\cb,\cc)$ such that $A$ is the largest divisor of $m_{\I\I\I}m_{\I\rV}$ dividing $\chi$ and $m_{\I\I\I}m_{\I\rV}\mid \Delta_2(\chi)$ and $p^{2n+k_p}\mid\mid \Delta_2(\chi)$ for every prime $p$ with $T_p = I_{2n}$ or $T_p = T_{2n+1}$ is
$$\ll_\epsilon  \frac{N^{1+\epsilon}}{m_{\I\I\I}m_{\I\rV}Y^{1/3}} \frac{N^{1+\epsilon}}{AY^{1/2}} \prod_{T_p = I_{2n}\,or\,I_{2n+1}} p^{-n-k_p/2}.$$
The number of choices for $A$ and the $k_p$'s is $\ll Y^\epsilon$. Combining with the bound \eqref{eq:mj1}, we have
$$\frac{Y}{N^3}\sum_{\substack{\cb\ll N^{1+\epsilon}/Y^{1/3}\\\cc\ll N^{1+\epsilon}/Y^{1/2}}}
\hat{\phi}(0,\cb,\cc) \ll_\epsilon \frac{Y^{1/6+\epsilon}}{N}m_{\I\I\I}\prod_{T_p = I_{2n}} p^{2n}\prod_{T_p = I_{2n+1}} p^{2n+1} = \frac{Y^{1/6+\epsilon}}{m_{\I\I\I}m_{\I\rV}^2}.$$
To bound the sum of $\hat{\phi}(\ca,\cb,0)$, we need a slight refinement. Fix again a divisor $A$ of $m_{\I\I\I}m_{\I\rV}$ and a nonnegative integer $\ell_p$ for every prime $p$ with $T_p = I_{\geq2}$. Suppose $\chi = (\ca,\cb,0)$ with $p^{n+\ell_p}\mid\mid \cb$ for every prime $p$ with $T_p = I_{2n}$ or $T_p = T_{2n+1}$. In order for $\widehat{\Phi_{T_p}}(\chi)\neq 0$ at these primes $p$, we need also $p^{n+\ell_p}\mid\ca$ by Lemma \ref{lemequi2}. If we further require that $A$ is the largest divisor of $m_{\I\I\I}m_{\I\rV}$ dividing $\chi$ and $m_{\I\I\I}m_{\I\rV}\mid \Delta_2(\chi)$, then the number of such $\chi$ is
$$\ll_\epsilon  \frac{N^{1+\epsilon}}{AY^{1/6}}\frac{N^{1+\epsilon}}{m_{\I\I\I}m_{\I\rV}Y^{1/3}} \prod_{T_p = I_{2n}\,or\,I_{2n+1}} p^{-2n-2\ell_p}$$
and for any such $\chi$, we have
$$\hat{\phi}(\chi) \ll A \, m_{\I\I\I}^2 \, m_{\I\rV} \prod_{T_p = I_{2n}} p^{3n+\ell_p} \prod_{T_p = I_{2n+1}} p^{3n+1+\ell_p}.$$
Combining these two bounds gives
$$\frac{Y}{N^3}
\sum_{\substack{\ca\ll N^{1+\epsilon}/Y^{1/6}\\\cb\ll N^{1+\epsilon}/Y^{1/3}}}
\hat{\phi}(\ca,\cb,0) \ll_\epsilon \frac{Y^{1/2+\epsilon}}{N}m_{\I\I\I}\prod_{T_p = I_{2n}} p^{n}\prod_{T_p = I_{2n+1}} p^{n+1} = \frac{Y^{1/2+\epsilon}}{Qm_{\I\rV}}.$$
Hence, we have
\begin{equation}\label{eqequisac}
S_{\ca\cc=0} \ll_\epsilon \frac{Y^{1/6+\epsilon}}{m_{\I\I\I}m_{\I\rV}^2} + \frac{Y^{1/2+\epsilon}}{Qm_{\I\rV}}.
\end{equation}

Next we consider $S_{\Delta_2=0}$. Third, we consider $S_{\Delta_2=0}$. Note that every $\chi\in
\widehat{U(\Z)}$ with $\Delta_2(\chi)=0$ is of the form
$(\alpha^2,\alpha\beta,\beta^2)$ with $\alpha,\beta\in\Z$. Fix a divisor $A$ of $m_{\I\I\I}m_{\I\rV}$ and nonnegative integers $\ell_p$ for each prime $p$ with $T_p=I_{\geq2}$. Suppose $A$ is the largest divisor of $m_{\I\I\I}m_{\I\rV}$ dividing $\chi$ and $p^{\ell_p}\mid\mid\beta$ for all $p$ with $T_p = I_{\geq2}$. Then similar to the case of $\hat{\phi}(\ca,\cb,0)$, we also need $p^{\ell_p}\mid\alpha$ in order for $\widehat{\Phi_{T_p}}(\chi)\neq0$, in which case
$$\hat{\phi}(\chi) \ll A\, m_{\I\I\I}^2\, m_{\I\rV} \prod_{T_p=I_{2n}} p^{3n+\ell_p} \prod_{T_p = T_{2n+1}}p^{3n+1+\ell_p}.$$
The number of such characters $\chi$ is
$$\ll_\epsilon \frac{N^{1/2+\epsilon}}{AY^{1/12}}\frac{N^{1/2+\epsilon}}{AY^{1/4}}\prod_{T_p = I_{2n}\,or\,I_{2n+1}} p^{-2\ell_p}.$$
Combining these two bounds gives
\begin{equation}\label{eqequidelta}
S_{\Delta_2=0}\ll_\epsilon \frac{Y^{2/3+\epsilon}}{N^2}m_{\I\I\I}^2\, m_{\I\rV}\prod_{T_p=I_{2n}} p^{3n} \prod_{T_p = T_{2n+1}}p^{3n+1}=\frac{Y^{2/3+\epsilon}}{Qm_\odd m_{\I\I\I}m_{\I\rV}^2}.
\end{equation}

Finally, we turn to $S_{\neq 0}$. Once agian, we fix a divisor $A$ of $m_{\I\I\I}m_{\I\rV}$ and a nonnegative integer $k_p$ for each prime $p$ with $T_p = I_{\geq2}$. The number of characters $\chi=(\ca,\cb,\cc)$ such that $A$ is the largest divisor of $m_{\I\I\I}m_{\I\rV}$ dividing $\chi$, $m_{\I\I\I}m_{\I\rV}\mid \Delta_2(\chi)$ and $p^{2n+k_p}\mid\mid\Delta_2(\chi)$ for any prime $p$ with $T_p = I_{2n}$ or $T_p = T_{2n+1}$ is
$$\ll_\epsilon Y^\epsilon \frac{N^{1+\epsilon}}{AY^{1/3}} \frac{N^{2+\epsilon}}{m_{\I\I\I}m_{\I\rV}Y^{2/3}}\prod_{T_p = I_{2n}\,or\,I_{2n+1}}p^{-2n-k_p}.$$
Indeed, the above bounds the number of pairs $(\cb,\Delta_2(\chi))$ satisfying the desired divisibility conditions, and given $\cb$ and $\Delta_2(\chi)$, there are $Y^\epsilon$ choices for $\ca$ and $\cc$. Combining with \eqref{eq:mj1} then gives
\begin{equation}\label{eqequineq0}
S_{\neq0} \ll_\epsilon Y^\epsilon m_{\I\I\I}\prod_{T_p=I_{2n}}p^n\prod_{T_p=I_{2n+1}}p^{n+1} = \frac{Qm_\odd}{m_{\I\rV}}Y^\epsilon.
\end{equation}

Theorem \ref{thm:equimain} now follows from \eqref{eqequibase}, \eqref{eqequis0}, \eqref{eqequisac},
\eqref{eqequidelta}, \eqref{eqequineq0}, and the AM-GM inequality.
\section{The family of cubic fields with prescribed shapes}\label{seccubicfields}

A cubic ring is a commutative ring with unit that is free
of rank 3 as a $\Z$-module. Given a cubic ring $R$, the {\it trace}
$\Tr(\alpha)$ of an element $\alpha\in R$ is the trace of the linear
map $\times\alpha:R\to R$. The {\it discriminant} $\Disc(R)$ of $R$ is
then the determinant of the bilinear pairing
$$R\times R\to\Z,\quad\quad (\alpha,\beta):=\Tr(\alpha\beta).$$ Given
a {\it nondegenerate} cubic ring $R$, i.e., a cubic ring $R$ with
nonzero discriminant, we then consider the cubic etal\'e algebras
$R\otimes\Q$ over $\Q$ and $R\otimes\R$ over $\R$. There are two
possibilities for $R\otimes\R$, namely, $\R^3$ and $\R\oplus\C$. We
have $R\otimes\R\cong\R^3$ when $\Disc(R)>0$ (equivalently, when the
signature of $R\otimes\Q$ is $(3,0)$) and $R\otimes\R\cong\R\oplus\C$
when $\Disc(R)<0$ (equivalently, when the signature of $R\otimes\Q$ is
$(1,2)$).

The ring $R$ embeds as a lattice into $R\otimes\R$ with covolume
$\sqrt{\Disc(R)}$. As regarded as this lattice, the element $1\in R$
is part of any Minkowski basis, and so the first successive minima of
$R$ is simply $1$. Let $\ell_1(R)\leq \ell_2(R)$ denote the other two
successive minima of $R$. We define the {\it skewness} of $R$ by
$$\sk(R):=\ell_2(R)/\ell_1(R).$$

Given a field $K$, we denote the ring of integers of $K$ by $\cO_K$,
and the class group of $K$ by $\Cl(K)$. For positive real numbers $X$
and $Z$, let $\RR_3^\pm(X,Z)$ denote the set of cubic fields $K$ that
satisfy the following two bounds: $X\leq\pm\Disc(\cO_K)<2X$ and
$\sk(\cO_K)>Z$. Set $\RR_3(X,Z)$ to be the union $\RR_3^+(X,Z)\cup
\RR_3^-(X,Z)$. In this section, we prove the following result.
\begin{theorem}\label{thm:skew}
Let $X$ and $Z$ be positive real numbers. Then
$$\sum_{K\in\RR_3(X,Z)}|\Cl(K)[2]| \ll X/Z,$$
where the implied constants are independent of $X$ and $Z$.
\end{theorem}

This section is organized as follows. In Section \ref{sec:4.1} we
recall the parametrization of cubic rings and of $2$-torsion elements
in the class groups of cubic rings, in terms of integral orbits for
the action of $\GL_2(\Z)$ on $\Sym^3(\Z^2)$ and of
$\GL_2(\Z)\times\SL_3(\Z)$ on $\Z^2\otimes\Sym^2(\Z^3)$,
respectively. In \S\ref{sec:4.2} and \S\ref{sec:4.3}, we then prove Theorem \ref{thm:skew} using these parametrizations
in conjunction with geometry-of-numbers methods.

\subsection{The parametrization of cubic rings
  and the $2$-torsion in their class groups}\label{sec:4.1}

In this section, we recall two parametrizations. First, the
parametrization of cubic rings, due to Levi \cite{levicubicparam},
Delone--Faddeev \cite{delonefaddeev}, and Gan--Gross--Savin
\cite{gangrosssavin}, and second, Bhargava's parametrization
\cite{hcl2} of elements in the $2$-torsion subgroups of cubic rings.
Let $V_3=\Sym^3(2)$ denote the space of binary cubic forms. We
consider the {\it twisted action} of $\GL_2$ on $V_3$ given by
\begin{equation*}
(\gamma\cdot f)(x,y):=\frac{1}{\det\gamma}f((x,y)\cdot\gamma),
\end{equation*}
for $\gamma\in\GL_2$ and $f(x,y)\in V_3$. Then we have the following
result.
\begin{theorem}[\cite{levicubicparam,delonefaddeev,gangrosssavin}]\label{thldf}
There is a natural bijection between the set of $\GL_2(\Z)$-orbits on
$V_3(\Z)$ and the set of cubic rings.
\end{theorem}

We collect some well known facts about the above bijection (for proofs
and a more detailed discussion, see \cite[\S 2]{bstcubic}). For an
integral binary cubic form $f$, we denote the corresponding cubic ring
by $R_f$. The bijection is discriminant preserving, i.e., we have
$\Delta(f)=\Disc(R_f)$. The ring $R_f$ is an integral domain if and
only if $f$ is irreducible over $\Q$. The group of automorphisms of
$R_f$ is isomorphic to the stabilizer of $f$ in $\GL_2(\Z)$.

The bijection of Theorem \ref{thldf} can be explicitly described as
follows: given a cubic ring $R$, consider the map
$R/\Z\to\wedge^2(R/\Z)\cong\Z$ given by $r\mapsto r\wedge r^2$. This
map is easily seen to be a cubic map and gives the binary cubic form
corresponding to $R$.  In fact, this map yields the finer bijection
\begin{equation}\label{eqfinbijcub}
  V_3(\Z) \longleftrightarrow \{(R,\omega,\theta)\},
\end{equation}
where $R$ is a cubic ring and $\langle\omega,\theta\rangle$ is a basis
for the 2-dimensional $\Z$-module $R/\Z$. The integral
binary cubic form corresponding to $(R,\omega,\theta)$ is $f(x,y)$,
where
\begin{equation}\label{eqrefbijcub}
(x\omega+y\theta)\wedge(x\omega+y\theta)^2=f(x,y)(\omega\wedge\theta).
\end{equation}
It is easily seen that the actions of $\GL_2(\Z)$ on $V_3(\Z)$
and on the set of triples $(R,\omega,\theta)$ agree. Here the latter
action is given simply by the natural action of $\GL_2(\Z)$ on the
basis $\{\omega,\theta\}$ of $R/\Z$.

Let $f$ be an integral binary cubic form, and let $(R,\omega,\theta)$
be the corresponding triple. Fix an element $\alpha=n+a\omega+b\theta$
of $R$, where $n$, $a$, and $b$ are integers and $(a,b)\neq(0,0)$. The
ring $\Z[\alpha]$ is a subring of $R$ having finite index denoted
$\ind(\alpha)$. It follows from \eqref{eqrefbijcub} that we have
\begin{equation}\label{eqindalpha}
\ind(\alpha)=f(a,b).
\end{equation}
Clearly $\ind(\alpha)=\ind(\alpha+n)$ for $n\in\Z$. Finally, we note
that the bijections of Theorem \ref{thldf} and \eqref{eqrefbijcub}
continue to hold if $\Z$ is replaced by any principal ideal domain
\cite[Theorem 5]{bswglobal1}.

\vspace{.1in}

Next, we describe the parametrization of $2$-torsion ideals in the
class groups of cubic rings. Let $W$ denote the space
$2\otimes\Sym^2(3)$ of pairs of ternary quadratic forms. For a ring
$S$, we write elements $(A,B)\in W(S)$ as a pair of $3\times
3$ symmetric matrices with coefficients in $S$. The group
$G_{2,3}=\GL_2\times\SL_3$ acts on $W$ via the action
$(\gamma_2,\gamma_3)\cdot
(A,B):=(\gamma_3A\gamma_3^t,\gamma_3B\gamma_3^t)\gamma_2^t$. We have
the {\it resolvent map} from $W$ to $V_3$ given by
\begin{equation*}
\begin{array}{rcl}
W&\to& V_3\\
(A,B)&\mapsto&\det(Ax+By).
\end{array}
\end{equation*}
The resolvent map respects the group actions on $W$ and $V_3$: we have
\begin{equation}\label{eqres}
\Res((\gamma_2,\gamma_3)\cdot
(A,B))=(\det\gamma_2)\cdot\Res(A,B).
\end{equation}

The following result parametrizing $2$-torsion ideals in cubic rings
is due to Bhargava \cite[Theorem 4]{hcl2}.
\begin{theorem}[\cite{hcl2}]\label{thbh2tor}
There is a bijection between $\GL_2(\Z)\times\SL_3(\Z)$-orbits on
$W(\Z)$ and equivalence classes of triples $(R,I,\delta)$, where $R$
is a cubic ring, $I\subset R$ is an ideal of $R$ having rank-3 as a
$\Z$-module, and $\delta$ is an invertible element of $R\otimes\Q$
such that $I^2\subset (\delta)$ and $N(I)^2=N(\delta)$. Here two
triples $(R,I,\delta)$ and $(R',I',\delta')$ are equivalent if there
exists an isomorphism $\phi:R\to R'$ and an element $\kappa\in
R\otimes\Q$ such that $I'=\phi(\kappa I)$ and
$\delta'=\phi(\kappa^2\delta)$. Moreover, the ring $R$ of the triple
corresponding to a pair $(A,B)$ is the cubic ring corresponding to
$\Res(A,B)$ under the Delone--Faddeev parametrization.
\end{theorem}

When $R=R_f$ is the maximal order in a cubic field $K$, the above
result gives a bijection between the set of
$\GL_2(\Z)\times\SL_3(\Z)$-orbits on the set of pairs $(A,B)\in W(\Z)$
with resolvent $f$, and the set of equivalence classes of pairs
$(I,\delta)$, where $I$ is an ideal of $R$, $\delta\in K$ and
$I^2=(\delta)$. This latter set is termed the {\it $2$-Selmer group}
$\Sel_2(K)$ of $K$ (see Definition 5.2.4 and Proposition 5.2.8 of
\cite{cohenadvancedtopics}) and fits into the exact sequence
\begin{equation*}
1\to R^\times/(R^\times)^2\to\Sel_2(K)\to\Cl(K)[2]\to 1,
\end{equation*}
where $R^\times$ denotes the unit group of $K$. Thus bounds on the
$2$-Selmer group of $K$ directly imply bounds on the
$2$-torsion subgroup of the class group of $K$.

\subsection{The number of cubic fields with bounded discriminants and
  skewed rings of integers}\label{sec:4.2}

The goal of this section is to prove the following result.

\begin{proposition}\label{propfewskf}
  Let $X$ and $Z$ be positive real numbers. There exists some constant
  $C$ such that $\RR^\pm_3(X,Z)$ is empty if $Z>CX^{1/6}$. Otherwise
  $|\RR^\pm_3(X,Z)|=O(X/Z)$.
\end{proposition}

For any subset $S$ of
$V_3(\R)$, let $S^\pm$ denote the set of elements $f$ such that
$\pm\Delta(f)>0$. Then $V_3(\R)^+$ (resp.\ $V_3(\R)^-$) consists of a
single $\GL_2(\R)$-orbit and corresponds to the cubic algebra $\R^3$
(resp.\ $\R\oplus\C$). We denote this cubic $\R$-algebra by $R^\pm$.
Let $\FF_2$ denote Gauss' fundamental domain for the action of
$\GL_2(\Z)$ on $\GL_2(\R)$. We write elements of $\GL_2(\R)$ in
Iwasawa coordinates, in which case we have
\begin{equation*}
\FF_2=\{n\alpha k\lambda:n\in N'(t),\alpha(t)\in A',k\in K,\lambda\in\Lambda\},
\end{equation*}
where,
\begin{equation}\label{nak}
N'(t)= \left\{\left(\begin{array}{cc} 1 & {} \\ {u} & 1 \end{array}\right):
        u\in\nu(t) \right\}    , \;\;
A' = \left\{\left(\begin{array}{cc} t^{-1} & {} \\ {} & t \end{array}\right):
       t\geq \sqrt[4]3/\sqrt2 \right\}, \;\;
\Lambda = \left\{\left(\begin{array}{cc} \lambda & {} \\ {} & \lambda
        \end{array}\right):
        \lambda>0 \right\},
\end{equation}
and $K$ is as usual the (compact) real orthogonal group ${\rm
  SO}_2(\R)$; here $\nu(t)$ is a union of one or two subintervals of
$[-\frac12,\frac12]$ depending only on the value of $t$.  Elements
$n\alpha(t) k\lambda$ are expressed in their Iwasawa coordinates as
$(n,t,\lambda,k)$.  Fix compact sets $B^\pm\subset V_3(\R)^\pm$ that
are closures of open bounded sets. Then for every point $v\in B^\pm$,
the set $\FF_2\cdot v$, viewed as a multiset, is a cover of a
fundamental domain for the action of $\GL_2(\Z)$ on $V_3(\R)^\pm$ of
absolutely bounded degree.  Recall that for a cubic ring $R$, its
skewness $\text{sk}(R)$ is defined to be the quotient
$\ell_2(R)/\ell_1(R)$ where $1,\ell_1(R),\ell_2(R)$ are the successive
minima of $R$, regarded as a lattice inside $R\otimes\R$. We have the
following lemma.
\begin{lemma}\label{lemsk}
Let $v\in B^\pm$ be any binary cubic form. Let
$\gamma=(n,t,\lambda,k)\in\FF_2$ be such that $f=\gamma\cdot v$ is an
integral binary cubic form. Then we have
\begin{equation*}
\sk(R_f)\asymp t^2
\end{equation*}
where $R_f$ denotes the cubic ring corresponding to $f$.
\end{lemma}
\begin{proof}
Every binary cubic form $v$ in $V_3(\R)^\pm$ gives rise to the cubic
algebra $R^\pm$, where $R^+\cong \R^3$ and $R^-\cong\C\oplus\R$, along
with elements $\alpha_v$ and $\beta_v$ such that $\langle
1,\alpha_v,\beta_v\rangle$ form a basis for $R^\pm$. Furthermore, the
lattice spanned by $1$, $\alpha_v$, and $\beta_v$ has covolume
$\sqrt{|\Delta(v)|}$. Since $B^\pm$ is compact, it follows that we
have $|\alpha_v|\cdot|\beta_v|\ll\sqrt{|\Delta(v)|}$ for $v\in B^\pm$.
Additionally, the action of $\GL_2(\R)$ on $V_3(\R)$ agrees with the
action of $\GL_2(\R)$ on pairs $(\alpha_v,\beta_v)$ by linear change
of variables. That is, we have $(\alpha_{\gamma\cdot
  v},\beta_{\gamma\cdot v})=\gamma\cdot(\alpha_v,\beta_v)$.

Let $f=\gamma\cdot v$ be an integral binary cubic form as in the
statement of the lemma. Since $\gamma\in\FF_2$, it follows that
$|\alpha_f|\asymp \lambda t^{-1}$ and $|\beta_f|\asymp \lambda t$.  As
a consequence,
$|\alpha_f|\cdot|\beta_f|\asymp\sqrt{\Disc(f)}$. Therefore, we have
$\ell_2(R_f)/\ell_1(R_f)\asymp |\beta_f|/|\alpha_f|\asymp t^2$ as
necessary.
\end{proof}

Next, we have the following lemma, due to Davenport
\cite{davenport-volume1}, that estimates the number of lattice points
within regions of Euclidlean space.
\begin{proposition}[\cite{davenport-volume1}]\label{davlem}
  Let $\mathcal R$ be a bounded, semi-algebraic multiset in $\R^n$
  having maximum multiplicity $m$, and that is defined by at most $k$
  polynomial inequalities each having degree at most $\ell$.
  Then the number of integral lattice points $($counted with
  multiplicity$)$ contained in the region $\mathcal R$ is
\[\Vol(\mathcal R)+ O(\max\{\Vol(\bar{\mathcal R}),1\}),\]
where $\Vol(\bar{\mathcal R})$ denotes the greatest $d$-dimensional
volume of any projection of $\mathcal R$ onto a coordinate subspace
obtained by equating $n-d$ coordinates to zero, where
$d$ takes all values from
$1$ to $n-1$.  The implied constant in the second summand depends
only on $n$, $m$, $k$, and $\ell$.
\end{proposition}

We are now ready to prove Proposition \ref{propfewskf}.

\vspace{.1in}

\noindent\textbf{Proof of Proposition \ref{propfewskf}:} A general
version of the first claim of the proposition, applying to number
fields of all degrees, is obtained in \cite[Theorem 3.1]{bstttz}, and
further generalizations are proved in \cite{Lapierrethesis}. For
completeness, we include a proof for our case below. Let $K$ be a
cubic field whose ring of integers $\cO_K$ belongs to $\RR^\pm_3(X,Z)$,
and let $\langle 1,\alpha,\beta\rangle$ be a Minkowski basis for
$\cO_K$ with $|\alpha|\leq|\beta|$. Consider the ring $\Z[\alpha]$
which is a suborder of $\cO_K$. We have
\begin{equation*}
X^{1/2}\asymp\sqrt{\Disc(\cO_K)}\ll\sqrt{\Disc(\Z[\alpha])}\ll|\alpha|^3,
\end{equation*}
and it follows that $|\alpha|\gg X^{1/6}$. Since $|\alpha||\beta|\asymp
X^{1/2}$, we have $Z=|\beta|/|\alpha|\ll X^{1/2}/X^{2/6}=X^{1/6}$ and the
first claim of the proposition follows.

We now estimate $|\RR^\pm_3(X,Z)|$ under the assumption that $Z\ll
X^{1/6}$ following the setup of \cite[\S5]{bstcubic}. Let $v$ be an
element of the compact set $B^\pm$. If $(n,t,\gamma,k)\cdot v$
corresponds to an cubic ring $R$ with $X\leq\Disc(R)<2X$ and
$\sk(R)>Z$, then it follows that $\lambda\asymp X^{1/4}$ and $t\gg
Z^{1/2}$, respectively, where the latter fact follows from Lemma
\ref{lemsk}. Hence we have
\begin{equation*}
\begin{array}{rcl}
\displaystyle |\RR^\pm_3(X,Z)|&\leq&
\displaystyle\int_{\substack{g=(n,t,\lambda,k)\in\FF_2\\\lambda\asymp X^{1/4}\\t\gg Z^{1/2}}}
\#\{g\cdot B^\pm\cap V_3(\Z)^\irr\}dg\\[.2in]&\leq&
\displaystyle
\int_{\substack{g=(n,t,\lambda,k)\in\FF_2\\\lambda\asymp X^{1/4}\\Z^{1/2}\ll t\ll X^{1/12}}}
\#\{g\cdot B^\pm\cap V_3(\Z)\}dg,
\end{array}
\end{equation*}
where the second inequality follows from the observation that if $t
\gg X^{1/12}$, then every element $f(x,y)$ in $g\cdot B^\pm$ has
$x^3$-coefficient less than $1$ in absolute value. Therefore no such
integral element $f(x,y)$ can be irreducible since its
$x^3$-coefficient must be $0$. Applying Proposition \ref{davlem} on
the set $g\cdot B^\pm$, we obtain
\begin{equation*}
\begin{array}{rcl}
|\RR^\pm_3(X,Z)|&\ll&
\displaystyle\int_{\lambda\asymp X^{1/4}}\int_{Z^{1/2}\ll t\ll X^{1/12}}
(\lambda^4+\lambda^3t^3)t^{-2}d^\times td^\times\lambda
\\[.2in]&\ll&\displaystyle\frac{X}{Z}+X^{5/6}\ll\frac{X}{Z},
\displaystyle
\end{array}
\end{equation*}
since $Z\ll X^{1/6}$. The proposition follows. $\Box$

\vspace{.1in}

We end this subsection with a counting result on the number of
primitive algebraic integers in a cubic field of bounded size to be
used in Section \ref{sec:ellipcount}. We say an element $\alpha$ in a
ring $R$ is {\it primitive} if $\alpha\neq n\beta$ for any $\beta\in
R$ and any integer $n\geq2$. We use the superscript $\Tr=0$ to denote
the subset of elements of trace $0$.

\begin{lemma}\label{lemkalcountel}
Let $K$ be a cubic field with discriminant $D$. For any real number
$Y>0$, let $N_K(Y)$ denote the number of primitive elements
$\alpha\in\cO_K^{\Tr=0}$ with $|\alpha|<Y$. Then
\begin{equation}\label{lem321eqt}
  N_K(Y)\leq\left\{
  \begin{array}{ccl}
    0 &\mbox{if}& Y<\ell_1(K);\\[.1in] 1 &\mbox{if}& \ell_1(K)\leq
    Y<\ell_2(K);\\[.1in]
    \frac{Y^2}{\sqrt{D}}+O\bigl(\frac{Y}{\ell_1(K)}\bigr) &\mbox{if}&
    \ell_2(K)\leq Y.
  \end{array}\right.
\end{equation}
\end{lemma}

\noindent Note that if $\ell_2(K)\leq Y$, then
$\frac{Y}{\ell_1(K)}\ll\frac{Y^2}{\sqrt{D}}$ and so we simply have
$N_K(Y)\ll\frac{Y^2}{\sqrt{D}}$, which is the best possible bound in
this case.
\vspace{.1in}

\begin{proof}
The first two lines of \eqref{lem321eqt} are clearly true. In fact,
they are equalities. Then final claim follows from Proposition
\ref{davlem} by replacing $N_K(Y)$ by the overcount where we count all
(not merely primitive) traceless elements $\alpha\in\cO_K$, since
$\cO_K^{\Tr=0}$ considered as a lattice inside $(K\otimes\R)^{\Tr=0}$
has covolume $\sqrt{D}$.
\end{proof}

\subsection{The 2-torsion subgroups in the class groups
  of cubic fields}\label{sec:4.3}

Let $K$ be a cubic field, and let $f\in V_3(\Z)$ be the binary cubic
form corresponding to $\cO_K$, the ring of integers of $K$. A
consequence of Theorem \ref{thbh2tor} is that the set of $2$-torsion
elements in the class group of $K$ injects into the set of
$\SL_3(\Z)$-orbits on the elements $(A,B)\in W(\Z)$ satisfying
$\Res(A,B)=f$.

Choose Iwasawa coordinates $(n,t,\lambda,k_2)$ for $\GL_2(\R)$ as in
the previous subsection and $(u,s_1,s_2,k_3)$ for $\SL_3(\R)$ as in
\cite[\S2.1]{manjulcountquartic}. A Haar-measure for $\SL_3(\R)$ in
these coordinates is $s_1^{-6}s_2^{-6}dudk_3d^\times s_1d^\times s_2$.
Let $\FF_3$ denote a fundamental domain for the action of $\SL_3(\Z)$
on $\SL_3(\R)$, such that $\FF_3$ is contained within a standard
Seigel domain in $\SL_3(\R)$. Then $\FF_{2,3}:=\FF_2\times\FF_3$ is a
fundamental domain for the action of $G_{2,3}(\Z)$ on
$G_{2,3}(\R)$. There are four $G_{2,3}(\R)$-orbits having nonzero
discriminant on $W(\R)$, and we denote them by $W(\R)^{(i)}$, $1\leq
i\leq 4$. For each $i$, let $\B_i\subset W(\R)^{(i)}$ be compact sets,
which are closures of open sets, such that $\Res(\B_i)\subset B^+\cup
B^-$, where $B^+$ and $B^-$ are as in the previous subsection. For each
element $w\in\B_i$, the set $\FF_{2,3}\cdot w$ is a cover of a
fundamental domain for the action of $G_{2,3}(\Z)$ on $W(\R)^{(i)}$.
Let $\B$ denote the union of the $\B_i$.

Next, let $W(\Z)^\irr$ denote the set of elements $(A,B)\in W(\Z)$
such that the resolvent of $(A,B)$ corresponds to an integral domain,
and such that $A$ and $B$ have no common root in $\P^2(\Q)$. Elements
in $W(\Z)$ that are not in $W(\Z)^\irr$ are said to be {\it reducible}. Given
a reducible element $w$ with resolvent $f$, either $R_f$ is not an
integral domain or $w$ corresponds to the identity element in the
class group of $R_f$. We now have the following lemmas.
\begin{lemma}\label{lemcgskew}
Let $g=(g_2,g_3)$ be an element in $\FF_{2,3}$, where
$g_2=(n,t,k_2,\lambda)\in\FF_2$ and $g_3\in\FF_3$. Let $(A,B)$ be an
integral element in $g\cdot\B$ such that $\Res(A,B)=f$. Then we have
\begin{equation*}
\Delta(f)\asymp \lambda^{12};\quad\sk(R_f)\asymp t^2.
\end{equation*}
\end{lemma}
\begin{proof}
The lemma follows immediately from \eqref{eqres} in conjunction with
Lemma \ref{lemsk} and the fact that $\Delta$ is a degree-$4$
homogeneous polynomial in the coefficients of $V_3$.
\end{proof}

\begin{lemma}\label{lemcgred}
Let $(A,B)$ be an element in $W(\Z)$. Denote the coefficients of $A$
and $B$ by $a_{ij}$ and $b_{ij}$, respectively. If $\det(A)=0$ or
$a_{11}=b_{11}=0$, then $(A,B)$ is reducible.
\end{lemma}
\begin{proof}
If $\det(A)=0$, then the cubic resolvent of $(A,B)$ has
$x^3$-coefficient $0$, implying that $(A,B)$ is reducible. If
$a_{11}=b_{11}=0$ then $A$ and $B$ have a common zero in $\P^2(\Q)$,
implying that $(A,B)$ corresponds to the identity element in the class
group of $R_f$.
\end{proof}


We are now ready to prove the second claim of Theorem \ref{thm:skew}.

\vspace{.1in}

\noindent\textbf{Proof of Theorem \ref{thm:skew}:}
We follow the setup and methods of \cite{manjulcountquartic}. To begin with,
averaging over $w\in \B$ as in \cite[(6) and (8)]{manjulcountquartic}, we obtain
\begin{equation*}
\begin{array}{rcl}
&&\displaystyle\sum_{K\in\RR^\pm_3(X,Z)}(|\Cl(K)[2]|-1)\\[.1in]&\ll&
\displaystyle\int_{g\in\FF_{2,3}}
\bigl|\bigl\{w\in g\cdot\B\cap W(\Z)^\irr:K_{\Res(w)}\in\RR_3(X,Z)
\bigr\}\bigr|dg\\[.2in]
&\ll&
\displaystyle\int_{s_1,s_2,t\gg 1}
\bigl|\bigl\{w\in ((\lambda,t),(s_1,s_2))\cdot\B\cap W(\Z)^\irr:
K_{\Res(w)}\in \RR_3(X,Z)
\bigr\}\bigr|\frac{d^\times\lambda d^\times t d^\times s_1 d^\times s_2}{t^{2}
  s_1^{6}s_2^{6}},
\end{array}
\end{equation*}
where $K_f$ denotes the algebra $\Q\otimes R_f$ for an integral binary
cubic form $f$.

The action of an element
$((\lambda,t),(s_1,s_2))\in\FF_{2,3}$ on $W(\R)$ multiplies each
coordinate $c_{ij}$ of $W$ by a factor which we denote by
$w(c_{ij})$. For example, we have $w(a_{11})=\lambda
t^{-1}s_1^{-4}s_2^{-2}$. The volume of $\B$ is some positive constant,
and when $\B$ is translated by an element $((\lambda,t),(s_1,s_2))$,
the volume is multiplied by a factor of $\lambda^{12}$, the product of
$w(c_{ij})$ over all the coordinates $c_{ij}$. Furthermore, the
maximum of the volumes of the projections of
$((\lambda,t),(s_1,s_2))\cdot \B$ is
$$
\ll\prod_{c_{ij}\in S}w(c_{ij})=\prod_{c_{ij}\not\in S}\lambda^{12}w(c_{ij}),
$$ where $S$ denotes the set of coordinates $c_{ij}$ of $W(\R)$ such
that the length of the projection of $((\lambda,t),(s_1,s_2))\cdot \B$
onto the $c_{ij}$-coordinate is at least $\gg 1$.

For the set $((\lambda,t),(s_1,s_2))\cdot \B\cap W(\Z)^\irr$ to be
empty, it is necessary that the projection of
$\B':=((\lambda,t),(s_1,s_2))\cdot \B$ onto the $b_{11}$-coordinate is
$\gg 1$. Otherwise, every integral element of $\B'$ has
$a_{11}=b_{11}=0$, and is hence reducible by Lemma
\ref{lemcgred}. Similary, the projections of $\B'$ onto the $a_{13}$-
and $a_{22}$-coordinates are also $\gg 1$ (since otherwise every
integral element $(A,B)$ of $\B'$ satisfy $\det(A)=0$). Finally, for
$\B'\cap W(\Z)$ to contain an element whose resolvent cubic form
corresponds to a field in $\RR_3(X,Z)$, we must have $\lambda\asymp
X^{1/12}$ and $Z^{1/2}\ll t\ll X^{1/2}$ by Lemma \ref{lemcgskew}.

Therefore, applying Proposition \ref{davlem} to the sets
$((\lambda,t),(s_1,s_2))\cdot \B$, we obtain
\begin{equation*}\begin{array}{rcl}
\displaystyle\sum_{K\in\RR^\pm_3(X,Z)}(|\Cl(K)[2]|-1)&\ll&
\displaystyle\int_{\substack{\lambda,t,s_1,s_2\\\lambda\asymp X^{1/12}\\Z^{1/2}\ll t\ll X^{1/12}\\s_1,s_2\gg 1}}
\bigl(\lambda^{12}(1+w(a_{11})^{-1}+w(a_{11}a_{12})^{-1}\bigr)
\frac{d^\times\lambda d^\times t d^\times s_1 d^\times s_2}{t^{2}s_1^{6}s_2^{6}}
\\[.3in]&\ll&\displaystyle
\int_{\substack{\lambda,t,s_1,s_2\\\lambda\asymp X^{1/12}\\Z^{1/2}\ll t\ll X^{1/12}\\s_1,s_2\gg 1}}
\bigl(\lambda^{12}+s_1^4s_2^2t\lambda^{11}+s_1^4s_2^4t^2\lambda^{10}\bigr)
\frac{d^\times\lambda d^\times t d^\times s_1 d^\times s_2}{t^{2}s_1^{6}s_2^{6}}
\\[.35in]&\ll&\displaystyle
X/Z+X^{11/12}/Z^{1/2}+X^{5/6+\epsilon},
\end{array}\end{equation*}
which is sufficient since $Z\ll X^{1/6}$. Theorem
\ref{thm:skew} now follows from this bound and Proposition \ref{propfewskf}.

\section{Embedding into the space of binary quartic forms}\label{sec:Qinv}

Recall that $U_0(\Z)$ denotes the set of monic cubic
polynomials with zero $x^2$-coefficient, and $U_0(\Z)^\min$ denotes the set of elements $f(x)\in U_0(\Z)$
such that the elliptic curve $y^2=f(x)$ has minimal discriminant among
all its quadratic twists. We define the height function
$H:U_0(\Z)\to\R_{\geq 0}$ by
\begin{equation*}
H(x^3+Ax+B)=\max\{4|A|^3,27B^2\}.
\end{equation*}
For $f(x)\in U_0(\Z)$, we write $K_f = \Q[x]/(f(x))$, $R_f =
\Z[x]/(f(x))$, and let $\cO_f$ denote the maximal order in $K_f$. The
$Q$-invariant $Q(f)$ of $f$ is defined as the index of $R_f$ in
$\cO_f$, and $D(f)$ is defined to be the discriminant of
$K_f$. Observe from Table \ref{tabloc} that for primes $p$ of type
$\I\I\I$, $\I\rV$ and $\I_{2n+1}$, we have $p\mid Q(f)$ and $p\mid D(f)$. Note also $\gcd(Q(f),D(f))$ is squarefree.

In this section, we obtain a bound on the number of elements $f\in
U_0(\Z)^\min$, having bounded height, such that both $Q(f)$ and
$\gcd(Q(f),D(f))$ are large.

\begin{theorem}\label{qinvmt}
Let $Q$ and $q$ be positive real numbers with $Q\geq q$. Let $N_{Q,q}(Y)$ denote the number of
elements $f(x)\in U_0(\Z)^\min$ such that $H(f)<Y$, $|Q(f)|>Q$, and
$\gcd(Q(f),D(f))>q$. Then
\begin{equation*}
N_{Q,q}(Y)\ll_\epsilon
  \frac{Y^{5/6+\epsilon}}{qQ}
+ \frac{Y^{7/12+\epsilon}}{Q^{1/2}},
  \end{equation*}
  where the implied constant is independent of $Q$, $q$ and $Y$.
\end{theorem}

This section is organized as follows. First, in \S5.1, we collect
classical results on the invariant theory of the action of $\PGL_2$ on
the space $V_4$ of binary quartic forms, and summarize the reduction
theory of binary quartics developed in \cite{bs2sel}. Next, in \S5.2,
we restrict to the space $V_4(\Z)^\red$ of binary quartic
forms with a linear factor. We develop the invariant theory for the action of $\PGL_2$ on
this space, and construct an embedding $U_0(\Z)^\min\to V_4(\Z)^\red$.

In Sections 5.3, 5.4, and 5.5, we estimate the number of
$\PGL_2(\Z)$-orbits on elements in $V_4(\Z)^\red$ with bounded height
and large $Q$-invariant and whose $Q$- and $D$-invariants have a large common factor. We do this by fibering the
space $V_4(\Z)^\red$ by their roots in $\P^1(\Z)$. Given an element
$r\in\P^1(\Z)$, the set of elements in $V_4(\Z)$ that vanish on $r$ is a
lattice $\L_r$. We then count the number of elements in
$\L_r$, using the Ekedahl sieve to exploit the condition that $\gcd(Q,D)$
is large.

\subsection{The action of $\PGL_2$ on the space $V_4$ of binary quartic forms}
Let $V_4$ denote the space of binary quartic forms. The group $\PGL_2$
acts on $V_4$ as follows: given $\gamma\in\GL_2$ and $g(x,y)\in V_4$,
define
\begin{equation*}
(\gamma\cdot g)(x,y):=\frac{1}{(\det \gamma)^2}\,g((x,y)\cdot\gamma).
\end{equation*}
It is easy to check that the center of $\GL_2$ acts trivially. Hence this
action of $\GL_2$ on $V_4$ descends to an action of $\PGL_2$ on
$V_4$.

The ring of invariants for the action of $\PGL_2(\C)$ on $V_4(\C)$ is
freely generated by two elements, traditionally denoted by $I$ and
$J$. Explicitly, for $g(x,y)=ax^4+bx^3y+cx^2y^2+dxy^3+ey^4$, we have
\begin{eqnarray*}
I(g) &=& 12ae-3bd+c^2,\\
J(g) &=& 72ace+9bcd-27ad^2-27eb^2-2c^3.
\end{eqnarray*}





We collect results from \cite[\S2.1]{bs2sel} on the reduction theory
of integral binary quartic forms.  For $i=0,1,2$, we let
$V_4(\R)^{(i)}$ to be the set of elements in $V(\R)$ with nonzero
discriminant, $i$-pairs of complex conjugate roots, and $4-2i$ real
roots. Furthermore, we write $V_4(\R)^{(2)}=V_4(\R)^{(2+)}\cup
V_4(\R)^{(2-)}$ as the union of forms that are positive definite and
negative definite. The four sets $L^{(i)}$ for $i\in\{0,1,2+,2-\}$
constructed in \cite[Table 1]{bs2sel} satisfy the following two
properties: first, $L^{(i)}$ are fundamental sets for the action of
$\R_{>0}\cdot\PGL_2(\R)$ on $V_4(\R)^{(i)}$ where $\R$ acts via scaling, and second, the sets
$L^{(i)}$ are absolutely bounded. It follows that the sets
$R^{(i)}:=\R_{>0}\cdot L^{(i)}$ are fundamental sets for the action of
$\PGL_2(\R)$ on $V_4(\R)^{(i)}$, and that the coefficients of an
element $f(x,y)\in R^{(i)}$ with $H(f)=Y$ are bounded by $O(Y^{1/6})$.

For $A'$, $N'(t)$, and $K$ defined in \eqref{nak}, set
\begin{equation*}
\FF_0=\{n\alpha(t)k:n(u)\in N'(t),\alpha(t)\in A',k\in K\}.
\end{equation*}
Then $\FF_0$ is a fundamental domain for the left multiplication action of $\PGL_2(\Z)$ on
$\PGL_2(\R)$; and the multisets $\FF_0\cdot R^{(i)}$ are $n_i$-fold
fundamental domains for the action of $\PGL_2(\Z)$ on $V(\R)^{(i)}$,
where $n_0=n_{2\pm}=4$ and $n_1=2$. Let $S\subset V_4(\Z)^{(i)}=V_4(\Z)\cap V_4(\R)^{(i)}$ be
any $\PGL_2(\Z)$-invariant set. Let $N_4(S;X)$ denote the number of
$\PGL_2(\Z)$-orbits on $S$ with height bounded by $X$ such that each
orbit $\PGL_2(\Z)\cdot f$ is counted with weight $1/\#\Stab_{\PGL_2(\Z)}(f)$. Let
$G_0\subset\PGL_2(\R)$ be a nonempty open bounded $K$-invariant set,
and let $d\gamma=t^{-2}dnd^\times tdk$, for $\gamma=ntk$ in Iwasawa
coordinates, be a Haar-measure on $\PGL_2(\R)$. Then, identically as
in \cite[Theorem 2.5]{bs2sel}, we have the following result.
\begin{theorem}\label{thbqfcavg}
We have
\begin{equation*}
N_4(S;X)=\frac{1}{n_i\Vol(G_0)}
\int_{\gamma\in\FF_0}\#\{S\cap \gamma G_0\cdot R^{(i)}_X\}\,d\gamma,
\end{equation*}
where $R^{(i)}_X$ denotes the set of elements in $R^{(i)}$ with height
bounded by $X$, the volume of $G_0$ is computed with respect to
$d\gamma$, and for any set $T\subset V(\R)$, the set of elements in
$T$ with height less than $X$ is denoted by $T_X$.
\end{theorem}

Apart from its use in this section to obtain a bound on reducible
binary quartic forms, Theorem \ref{thbqfcavg} will also be used in
Section 7 to prove Theorem \ref{thmsel}.

\subsection{Embedding $U_0(\Z)^\min$ into the space of reducible binary quartics}

Let $f(x)=x^3+Ax+B$ be an element in $U_0(\Z)^\min$ with $Q(f)=n$.  From
Theorem \ref{propcasessp}, it follows that there exists an integer
$r$, defined uniquely modulo $n$, such that $f(x+r)$ is of the form
\begin{equation*}
f(x+r)=x^3+ax^2+bnx+cn^2.
\end{equation*}
Assume that we have picked $r$ so that $0\leq r<n$. The ring of
integers $\cO_f$ in $K_f=\Q[x]/f(x)$ corresponds, under the
Delone--Faddeev bijection, to the binary cubic form
\begin{equation*}
h(x,y)=nx^3+ax^2y+bxy^2+cy^3.
\end{equation*}

Elements in $U_0(\Z)^\min$ with $Q$-invariant $n$ thus correspond to
integral binary cubic forms that represent $n$. However, this latter
condition is difficult to detect, at least using geometry-of-numbers
methods. Instead, we embed the space of binary cubic forms into the
space $V_4(\Z)^\red$ of binary quartic forms with a linear factor over $\Q$ by multiplying
by $y$. In fact, we will replace $V_4(\Z)^\red$ with its (at most $4$
to $1$) cover $\widetilde{V}_4(\Z)$ consisting of pairs
$(g(x,y),[\alpha,\beta])$, where $g$ is a reducible binary quartic
forms and $[\alpha,\beta]$ is a root of $f$. Explicitly,
\begin{equation*}
\widetilde{V}_4(\Z):=\{(g(x,y),[\alpha,\beta]):0\neq g(x,y)\in
V_4(\Z)^\red,\;\alpha,\beta\in\Z,\;\gcd(\alpha,\beta)=1,\;g(\alpha,\beta)=0\}
\end{equation*}
This gives us the following map $\tilde{\sigma}:U_0(\Z)^\min\to
\widetilde{V}_4(\Z)$:

\begin{equation}\label{eqsig}
\begin{array}{rcccl}
  \tilde{\sigma}:U_0(\Z)^\min&\to& V_3(\Z)&\to& \widetilde{V}_4(\Z)\\[.1in]
  f(x)&\mapsto &h(x,y)&\mapsto&(yh(x,y),[1,0]).
\end{array}
\end{equation}
The group $\PGL_2(\Z)$ acts on $\widetilde{V}_4(\Z)$ via
\begin{equation*}
  \gamma\cdot(g(x,y), [\alpha:\beta]) =
  ((\gamma\cdot g)(x,y),[\alpha:\beta]\gamma^{-1});
\end{equation*}
this is an action since $(\gamma\cdot
g)((\alpha,\beta)\gamma^{-1})=g(\alpha,\beta)=0$.  Aside from the
classical invariants $I$ and $J$, this action has an extra invariant,
which we denote by $Q$, defined as follows. Given $(g,[\alpha:\beta])\in \widetilde{V}_4(\Z)$, let $h(x,y) = g(x,y)/(\beta x-\alpha y)$ be the associated binary cubic form and we define
\begin{equation}\label{eqQexp}
Q(g,[\alpha:\beta]) = h(\alpha,\beta),\qquad D(g,[\alpha:\beta])) = \Delta(h).
\end{equation}
The $Q$-, $D$-invariants and the discriminant are related by
$$\Delta(g) = Q(g,[\alpha:\beta])^2D(g,[\alpha:\beta]).$$

We now have the following result.

\begin{proposition}
There is an injective map
\begin{equation*}
  \sigma: U_0(\Z)^\min\rightarrow \widetilde{V}_4(\Z)\rightarrow \PGL_2(\Z)
  \backslash \widetilde{V}_4(\Z),
\end{equation*}
such that for every $f\in U_0(\Z)^\min$, we have
\begin{equation}\label{eq:cond}
  I(f)=I(\sigma(f));\qquad J(f)=J(\sigma(f);\qquad Q(f) = Q(\sigma(f));
  \qquad D(f)=D(\sigma(f)).
\end{equation}
\end{proposition}
\begin{proof}
The first three equalities of \eqref{eq:cond} can be checked by a
direct computation. The injectivity of $\sigma$ then follows from the
fact that $I(f)$ and $J(f)$ determine $f$. Finally, the last equality
of \eqref{eq:cond} can be directly obtained:
\begin{equation*}
D(f)=\Delta(f)/Q(f)^2=\Delta(\sigma(f))/Q(\sigma(f))^2=D(\sigma(f)),
\end{equation*}
where the second equality follows since
$(I(f),J(f))=(I(\sigma(f),J(\sigma(f)))$ and so $\Delta(f) = \Delta(\sigma(f))$.
\end{proof}

Therefore, to prove Theorem \ref{qinvmt}, it suffices to count
$\PGL_2(\Z)$-orbits $(g,[\alpha:\beta])$ in $\tilde{V}_4(\Z)$, such that both
$Q(g,[\alpha:\beta])$ and the radical $\rad(\gcd(Q(g,[\alpha:\beta]),D(g,[\alpha:\beta])))$ are large.

\subsection{Counting $\PGL_2(\Z)$-orbits on reducible binary quartic forms}

We use the setup of \cite[\S2]{bs2sel}, which is recalled
in \S5.1.
Since the sets $L^{(i)}$ are absolutely bounded, the coefficients of any
element in $R^{(i)}=\R_{>0}\cdot L^{(i)}$ having height $Y$ are
bounded by $O(Y^{1/6})$. Hence the same is true of every element in
$G_0\cdot R^{(i)}_Y$, as $G_0$ is a bounded set. The set
$V_4(\Z)^\red$ is not a lattice. To apply geometry-of-numbers methods,
we fiber it over the set of possible linear factors. We write
\begin{equation}\label{eqfiberroot}
V_4(\Z)^\red=\bigcup_{r=[\alpha:\beta]}\L_r,
\end{equation}
where $\alpha$ and $\beta$ are coprime integers and for
$r=[\alpha:\beta]\in\P^1(\Z)$, we define $\L_r$ to be the set of all
integral binary quartic forms $f$ such that $f(r)=0$.  From Theorem
\ref{thbqfcavg}, in conjunction with the injection $\sigma$ of \S5.2,
we have
\begin{equation}\label{eqnqqyb1}
N_{Q,q}(Y)
\ll
\displaystyle\sum_{r\in\P^1(\Z)}\int_{(ntk)\in\FF_0}
\#\bigl\{g\in\L_r\cap
(ntk)G_0R^{(i)}_Y:
Q(g)>Q,\,\rad(\gcd(Q(g),D(g)))>q\bigr\}\,t^{-2}dnd^\times tdk
\end{equation}

As $\gamma$ varies over $\FF_0$, the set $\gamma G_0R^{(i)}_Y$
becomes skewed. More precisely, if $\gamma=ntk$ in Iwasawa
coordinates, then the five coefficients $a$, $b$, $c$, $d$, and $e$,
of any element of $\gamma G_0R^{(i)}(Y)$ satisfy
\begin{equation}\label{eqrbqfcoeffbounds}
a\ll \frac{Y^{1/6}}{t^4};\quad
b\ll \frac{Y^{1/6}}{t^2};\quad
c\ll Y^{1/6};\quad
d\ll t^2Y^{1/6};\quad
e\ll t^4Y^{1/6}.
\end{equation}
Hence when $t\gg Y^{1/24}$, the $x^4$-coefficient of any integral
binary quartic form in $\gamma G_0R^{(i)}(Y)$ is $0$, forcing a root
at the point $[1,0]\in\P^1(\Z)$. Moreover we expect it to be rare that
such a binary quartic form has another integral root. In what follow,
we first consider the lattice $\L_{[1,0]}$ in \S5.4, and consider the
rest of the lattices in \S5.5.

\subsection{The contribution from the root $r = [1:0]$}

Let $g(x,y)=bx^3y+cx^2y^2+dxy^3+ey^4\in\L_{[1:0]}$ be an integral binary
quartic form. We write $Q(g)$ for $Q(g,[1:0])$ and $D(g)$ for $D(g,[1:0])$. Then we have $Q(g)=b$ and
$D(g)=\Delta(bx^3+cx^2y+dxy^3+ey^3)$, the discriminant of the binary
cubic form $g(x,y)/y$. Hence, if a fixed $t\geq 1$ contributes to the
estimate $N_{Q,q}(Y)$ in \eqref{eqnqqyb1}, then we must have
\begin{equation}\label{eq10tbound}
t\ll\frac{Y^{1/12}}{Q^{1/2}}.
\end{equation}
We now fiber over the $O(Y^{1/6}/t^2)$ choices for $b$. For each such
choice, we have $O(Y^\epsilon)$ possible squarefree divisors $m$ of $b$. Fix such
a divisor $m>q$ such that $\rad(\gcd(Q(g),D(g)))=m$. Then $m\mid D(g)$ which
implies that
\begin{equation*}
3c^2d^2-4c^3e\equiv 0\pmod{m}.
\end{equation*}
Thus, the residue class of $e$ modulo $m$ is determined by $c$ and
$d$, unless $m\mid c$.

From \eqref{eqrbqfcoeffbounds}, we see that the number of elements in
$\L_{[1:0]}\cap(ntk)G_0R^{(i)}_Y$ with $b$ and $m$ fixed as above is bounded by
\begin{equation*}
  O\Bigl(\frac{t^6Y^{1/2}}{m}+t^6Y^{1/3}\Bigr)=
  O\Bigl(\frac{t^6Y^{1/2}}{q}+t^6Y^{1/3}\Bigr),
\end{equation*}
where the second term deals with the case $q\gg Y^{1/6}$.
It therefore follows that the contribution to $N_{Q,q}(Y)$ in
\eqref{eqnqqyb1} from the root $r=[1:0]$ is bounded by
\begin{equation}\label{eqfinbound10root}
  \int_{t=1}^{Y^{1/12}/Q^{1/2}}\frac{Y^{1/6+\epsilon}}{t^2}
  \Bigl(\frac{t^6Y^{1/2}}{q}+t^6Y^{1/3}\Bigr)t^{-2}d^\times t
\ll_\epsilon \frac{Y^{5/6+\epsilon}}{qQ}+\frac{Y^{2/3+\epsilon}}{Q},
\end{equation}
which is sufficiently small. The contribution from the root $r=[0:1]$
can be identically bounded.

\subsection{The contribution from a general root $r=[\alpha:\beta]$
  with $\alpha\beta\neq 0$}

Write $r = [\alpha:\beta]$ where $\alpha,\beta$ are coprime integers
and $\alpha\beta\neq 0$. Throughout this section, we denote the
torus element in $\FF_0$ with entries $t^{-1}$ and $t$ by
$a_t$. We have the bijection
\begin{equation}\label{eqskbij}
\begin{array}{rcl}
  \theta_t:\{\L_r\cap a_tG_0\cdot R^{(i)}_Y\}&\longleftrightarrow&
  \{a_t^{-1}\L_r \cap G_0 \cdot R^{(i)}_Y\}\\[.1in]
  ax^4+bx^3y+cx^2y^2+dxy^3+ey^4&
  \longmapsto& t^4ax^4+t^2bx^3y+cx^2y^2+t^{-2}dxy^3+t^{-4}ey^4,
\end{array}
\end{equation}
which preserves the invariants $I$ and $J$.  Define $\tilde{V_4}(\R)$
to be the set of pairs $(g(x,y),r)$, where $g(x,y)\in V_4(\R)$ and
$r\in\R^2$ such that $g(r)=0$. We extend the definitions of the
$Q$- and $D$-invariants to the space $\tilde{V_4}(\R)$ via
\eqref{eqQexp}. Set $r_t:=r\cdot
a_t=[t^{-1}\alpha,t\beta]$. Then we have
\begin{equation*}
Q(g,r)=Q(\theta_t\cdot g,r_t),\qquad D(g,r)=D(\theta_t\cdot g,r_t).
\end{equation*}

Identifying the space of binary quartics with $\R^5$ via the
coefficients $(a,b,c,d,e)$, we write
$$a_t^{-1}\L_r = \text{diag}(t^4,t^2,1,t^{-2},t^{-4}) \cdot
\big((\alpha^4,\alpha^3\beta,\alpha^2\beta^2,\alpha\beta^3,\beta^4)^\perp\big),$$
where $(\alpha^4,\alpha^3\beta,\alpha^2\beta^2,\alpha\beta^3,\beta^4)^\perp$ is the sublattice of $\Z^5$ perpendicular to $(\alpha^4,\alpha^3\beta,\alpha^2\beta^2,\alpha\beta^3,\beta^4)$ with respect to the usual inner product on
$\R^5$.  Since $\alpha$ and $\beta$ are coprime, the following vectors
form an integral basis for $a_t^{-1}\L_r$:
\begin{equation*}
  w_1 = (t^4\beta,-t^2\alpha,0,0,0),\quad
  w_2 = (0,t^2\beta,-\alpha,0,0),\quad
  w_3 = (0,0,\beta,-t^{-2}\alpha,0),\quad
  w_4 = (0,0,0,t^{-2}\beta,-t^{-4}\alpha).
\end{equation*}
Define the vector $v_t$ to be
$v_t:=(t\beta,-t^{-1}\alpha)\in\R^2$. Then it is easy to see that the
lengths of $w_i$ are given by:
\begin{equation}\label{eqbasislength}
|w_1| =
t^3|v_t|,\quad |w_2| = t|v_t|,\quad |w_3| = t^{-1}|v_t|,\quad |w_4| =
t^{-3}|v_t|,
\end{equation}
The next lemma proves that this basis is
{\it almost-Minkowski}. That is, the quotients $\langle
w_i,w_j\rangle/(|w_i||w_j|)$, for $i\neq j$, are bounded from above by
a constant $c<1$ independent of $t$ and $r$.
\begin{lemma}
  For $i\neq j$, we have
  \begin{equation*}
\langle w_i,w_{j}\rangle \leq \frac12 |w_i||w_j|.
  \end{equation*}
\end{lemma}
\begin{proof}
The inner product $\langle w_i,w_j\rangle$ for $i<j$ is $0$ unless
$j=i+1$. In those three cases, we have
\begin{equation*}
\frac{\langle w_i,w_{j}\rangle}{|w_i||w_j|}=\frac{|\alpha\beta|}{t^{-2}\alpha^2+t^2\beta^2}\leq \frac12,
\end{equation*}
by the AM-GM inequality.
\end{proof}

We will represent elements in $a_t^{-1}\L_r$ by four-tuples
$(a_1,a_2,a_3,a_4)\in\Z^4$, where such a tuple corresponds to the element
$a_1w_1 + a_2w_2 + a_3w_3 + a_4w_4$. Then we have the following lemma.
\begin{lemma}\label{lemtupinv}
Let $g(x,y)$ be an element in $\L_r$, and let $a_t^{-1} g(x,y)$
correspond to the four-tuple $(a_1,a_2,a_3,a_4)$. Then we have
\begin{eqnarray*}
g(x,y) &=& (\beta x - \alpha y)(a_1 x^3 + a_2 x^2 y + a_3 x y^2 + a_4 y ^3),\\
Q(g,r) &=& a_1\alpha^3 + a_2\alpha^2\beta + a_3\alpha\beta^2 + a_4\beta^3,\\
D(g,r) &=& \Delta_3(a_1,a_2,a_3,a_4),
\end{eqnarray*}
where $\Delta_3(a_1,a_2,a_3,a_4)$ denotes the discriminant of the
binary cubic form with coefficients $a_i$.
\end{lemma}
The above lemma follows from a direct computation. Next, we determine
when an element $(a_1,a_2,a_3,a_4)$ has small length.
\begin{lemma}\label{lemtupbound}
Suppose $g\in a_t^{-1}\L_r$, corresponding to $(a_1,a_2,a_3,a_4)$,
belongs to $G_0\cdot R^{(i)}_Y$ for some $i$. Then
\begin{equation}\label{eq:bounda}
  a_1 \ll \frac{Y^{1/6}}{t^{3}|v_t|};
  \quad a_2 \ll \frac{Y^{1/6}}{t|v_t|};
  \quad a_3 \ll \frac{Y^{1/6}}{t^{-1}|v_t|};
  \quad a_4 \ll \frac{Y^{1/6}}{t^{-3}|v_t|}.
\end{equation}
\end{lemma}
\begin{proof}
Let $|\cdot|$ denote the length of a binary quartic form, where
$V_4(\R)$ has been identified with $\R^5$ in the natural way. Then for
$g$ to belong in $G_0\cdot R^{(i)}_Y$, it must satisfy $|g|\ll
Y^{1/6}$.  For any real numbers $a_1,a_2,a_3,a_4$, we compute
\begin{eqnarray*}
  |a_1w_1+a_2w_2+a_3w_3+a_4w_4|^2 &\geq& a_1^2|w_1|^2 +  a_2^2|w_2|^2 + a_3^2|w_3|^2 + a_4^2|w_4|^2 \\[.15in]
  &&-|a_1||a_2||w_1||w_2|-|a_2||a_3||w_2||w_3| - |a_3||a_4||w_3||w_4|\\[.15in]
&\geq& \frac{3 - \sqrt{5}}{4}( a_1^2|w_1|^2 +  a_2^2|w_2|^2 + a_3^2|w_3|^2 + a_4^2|w_4|^2).
\end{eqnarray*}
(Of course, the exact constant is not important.) Therefore in order for
$|a_1w_1+a_2w_2+a_3w_3+a_4w_4|\ll Y^{1/6}$, Equation \eqref{eq:bounda}
must be satisfied.
\end{proof}

We now have the following proposition bounding the number of elements
in $\L_r\cap a_tG_0\cdot R^{(i)}_Y$ whose $Q$- and $D$-invariants
share a large common factor.
\begin{proposition}\label{aster}
For $t\gg 1$, we have
\begin{equation}\label{eqaster}
\#\{g(x,y)\in\L_r\cap a_tG_0\cdot R^{(i)}_Y:\rad(\gcd(Q(g,r),D(g,r)))>q\}=
\left\{
\begin{array}{lll}
0&\mbox{if}\;\;\;|v_t|\gg Y^{1/6}\\[.1in]
\displaystyle O\Bigl(\frac{Y^{2/3+\epsilon}}{q|v_t|^4}+\frac{Y^{1/2+\epsilon}}{t|v_t|^3}\Bigr)&\mbox{otherwise}
\end{array}
\right.
\end{equation}
where the implied constant is independent of $r$, $t$, and $Y$.
\end{proposition}
\begin{proof}
Using the bijection \eqref{eqskbij} in conjunction with Lemmas
\ref{lemtupinv} and \ref{lemtupbound}, we see that it is enough to
prove that the number of four-tuples of integers $(a_1,a_2,a_3,a_4)$,
satisfying \eqref{eq:bounda} and
\begin{equation*}
\rad(\gcd(a_1\alpha^3 + a_2\alpha^2\beta + a_3\alpha\beta^2 + a_4\beta^3,\Delta_3(a_1,a_2,a_3,a_4)))>q,
\end{equation*}
is bounded by the right hand side of \eqref{eqaster}.  Suppose first
$|v_t|\gg Y^{1/6}$. Then any binary quartic form $g(x,y)$ represented
by the four-tuple $(a_1,a_2,a_3,a_4)$ satisfying \eqref{eq:bounda}
must have $a_1=a_2=0$. From Lemma \ref{lemtupinv}, it follows that
$D(g,r)=0$ and hence $\Delta(g)=0$. Since $G_0\cdot R^{(i)}_Y$
contains no point with $\Delta=0$, it follows that the intersection is
empty, proving the first part of the proposition.

The second part of the proposition is proved by using the Ekedhal
sieve as developed in \cite{manjul-geosieve}.  We carry the sieve out
in detail so as to demonstrate that the implied constant in \eqref{eqaster}
is indeed independent of $r$ and $t$. Define
\begin{equation*}
  T_{\alpha,\beta}(a_1,a_2,a_3) := \Delta_3(a_1\beta^3, a_2\beta^3,
  a_3\beta^3, -(a_1\alpha^3 + a_2\alpha^2\beta + a_3\alpha\beta^2)).
\end{equation*}
It is clear that if $m\mid Q(g,r)$ and $m\mid D(g,r)$ for any integer $m$, then $m\mid
T_{\alpha,\beta}(a_1,a_2,a_3)$.

First, we bound the number of triples $(a_1,a_2,a_3)$ satisfying
\eqref{eq:bounda} such that $T_{\alpha,\beta}(a_1,a_2,a_3)=0$. For a
fixed pair $(a_1,a_2)\neq (0,0)$, by explicitly writing out $T_{\alpha,\beta}(a_1,a_2,a_3)$, we see that there are at most three possible
values of $a_3$ with $T_{\alpha,\beta}(a_1,a_2,a_3)=0$. This gives a
bound of $O(Y^{1/3}/(t^4|v_t|^2))$ on the number of triples
$(a_1,a_2,a_3)$ with $T_{\alpha,\beta}(a_1,a_2,a_3)=0$. Multiplying
with the number of all possibilities for $a_4$, we obtain the bound
\begin{equation}\label{eqgeoTeq0}
O\Bigl(\frac{Y^{1/2}}{t|v_t|^3}\Bigr)
\end{equation}
on the number of four-tuples of integers $(a_1,a_2,a_3,a_4)$,
satisfying \eqref{eq:bounda} and $T_{\alpha,\beta}(a_1,a_2,a_3)=0$.

Next, we fiber over triples $(a_1,a_2,a_3)$ with
$T_{\alpha,\beta}(a_1,a_2,a_3)\neq 0$ and satisfying
\eqref{eq:bounda}. In this case, we have $(a_1,a_2)\neq (0,0)$. Hence by \eqref{eq:bounda}, we may assume $\alpha,\beta,t\ll
Y^{1/6}$. Hence the value of $T_{\alpha,\beta}(a_1,a_2,a_3)$ is
bounded by a polynomial in $Y$ of fixed degree. It follows that the
number of squarefree divisors of $T_{\alpha,\beta}(a_1,a_2,a_3)$ is bounded by
$O_\epsilon(Y^\epsilon)$. Fix one such divisor $m>q$. We now fiber over a positive squarefree integer $\delta \ll Y^{1/6}/(t^3|v_t|)$ such that $\text{rad}(\gcd(a_1,a_2)) = \delta$. The number of such possible $(a_1,a_2)$ is $$\ll \frac{1}{\delta^2}\frac{Y^{1/6}}{t^3|v_t|}\frac{Y^{1/6}}{t|v_t|}.$$ Fix any such pair. Let $a_3$ be any integer satisfying \eqref{eq:bounda} such that $T_{\alpha,\beta}(a_1,a_2,a_3)\neq 0$. Let $m_1 = \gcd(m,\delta)$ and let $m_2 = m/m_1 > q/\delta$. Then the polynomial $\Delta_3(a_1,a_2,a_3,a_4)$ is identically $0$ modulo $m_1$ and quadratic in $a_4$ modulo any prime factor of $m_2$. Hence the number of these quadruples with the extra condition that $m\mid\Delta_3(a_1,a_2,a_3,a_4)$ is
$$\ll\frac{1}{\delta^2}\frac{Y^{1/6}}{t^3|v_t|}\frac{Y^{1/6}}{t|v_t|}\frac{Y^{1/6}}{t^{-1}|v_t|}\left(\frac{1}{q/\delta}\frac{Y^{1/6}}{t^{-3}|v_t|}+1\right) \ll_\epsilon \frac{Y^{2/3}}{\delta q|v_t|^4} + \frac{Y^{1/2}}{\delta^2t^3|v_t|^3}.$$
Summing over $\delta$ and all possible divisors $m$ gives the bound
\begin{equation}\label{eqgeoTneq0}
O\Bigl(\frac{Y^{2/3+\epsilon}}{q|v_t|^4} + \frac{Y^{1/2+\epsilon}}{t^3|v_t|^3}\Bigr).
\end{equation}
The proposition now follows from \eqref{eqgeoTeq0} and
\eqref{eqgeoTneq0}.
\end{proof}

We now impose the condition on the $Q$-invariant. From Lemma
\ref{lemtupinv} and \eqref{eq:bounda}, we obtain
\begin{equation*}
Q< |Q(g,r)| = |a_1\alpha^3 + a_2\alpha^2\beta + a_3\alpha\beta^2 + a_4\beta^3| \ll Y^{1/6}|v_t|^2.
\end{equation*}
In conjunction with \eqref{eqnqqyb1} and the estimates of
Proposition \ref{aster}, this yields
\begin{equation*}
\begin{array}{rcl}
N_{Q,q}(Y)&\ll_\epsilon&
\displaystyle\sum_{k\ll\log Y}\int_{t\gg 1}\sum_{\substack{r=[\alpha:\beta]\\ 2^k<|v_t|\leq 2^{k+1}}}
\Bigl(\frac{Y^{2/3+\epsilon}}{q|v_t|^4}+\frac{Y^{1/2+\epsilon}}{t|v_t|^3}\Bigr)t^{-2}d^\times t
\\[.3in]
&\ll_\epsilon&
\displaystyle\sum_{k\ll\log Y}\int_{t\gg 1}
\Bigl(\frac{Y^{5/6+\epsilon}}{qQ}+\frac{Y^{7/12+\epsilon}}{tQ^{1/2}}\Bigr)t^{-2}d^\times t
\\[.3in]&\ll_\epsilon&
\displaystyle\frac{Y^{5/6+\epsilon}}{qQ}+\frac{Y^{7/12+\epsilon}}{Q^{1/2}}.
\end{array}
\end{equation*}
This concludes the proof of Theorem \ref{qinvmt}.

\section{Uniformity estimates}

In this section, we prove Theorems \ref{thunifsqi} and
\ref{thunifsmind}, the main uniformity estimates. First, in \S6.1, we
use the results of \S3 to prove Theorem \ref{thunifsqi}. Next, in
\S6.2, we combine the results of \S4 and \S5 in order to obtain
Theorem \ref{thunifsmind}.

\subsection{The family of elliptic curves with
  squarefree index}\label{sec:ellipcount}

Recall the family $\E$ defined in the introduction. The assumption
that elliptic curves $E\in\E$ satisfy $j(E)\leq\log(\Delta(E))$
implies the height bound $H(E)\ll\Delta(E)^{1+\epsilon}$.  Given
$E\in\E$, let $E:y^2=f(x)=x^3+Ax+B$ be the minimal Weierstrass model
for $E$. Given an etal\'e algebra $K$ over $\Q$ with ring of integers
$\cO_K$, let $\cO_K^{\Tr=0}$ denote the set of traceless integral
elements in $K$. Consider the map
\begin{equation*}
  \E\to \{(K,\alpha):K\mbox{ cubic algebra over }\Q,\;
  \alpha\in\cO_K^{\Tr=0}\}
\end{equation*}
sending $E:y^2=f(x)$ to the pair $(\Q[x]/f(x),x)$. This map is
injective since if $E$ corresponds to the pair $(K,\alpha)$, then
$y^2=N_{K/\Q}(x-\alpha)$ recovers $E$. In order to parametrize
elements in $\E_\sf$, we will instead use the following modified map:
\begin{equation}\label{eqEEsfparam}
\begin{array}{rcl}
\displaystyle \sigma: \E&\to&
\displaystyle\{(K,\alpha):K\mbox{ cubic \'{e}tale algebra over }\Q,\;
\alpha\in\cO_K^{\Tr=0}\}
\\[.1in]
\displaystyle E:y^2=f(x)&\mapsto&
\displaystyle (\Q[x]/f(x),{\rm Prim}(x)),
\end{array}
\end{equation}
where for $0\neq x\in\cO_K$, the element ${\rm Prim}(x)$ is the unique
primitive integer in $\cO_K$ which is a positive rational multiple of
$x$. Note that the map $\sigma$ restricted to $\E_\sf$ is injective due to the squarefree condition on $\Delta(E)/C(E)$ at primes at least $5$ and that $E$ has good reduction at $2$ and $3$.
We start with the following lemma.
\begin{lemma}
Let $E$ be an elliptic curve and let $\sigma(E)=(K,\alpha)$. Then
$|\Sel_2(E)|\ll_\epsilon
|\Cl(K)[2]|\cdot|\Delta(E)|^\epsilon$. Furthermore, if $E\in\E_\sf$, then
$|\alpha|\ll H(E)^{1/6}$.
\end{lemma}
\begin{proof}
The first claim is a direct consequence of \cite[Proposition
  7.1]{brumerkramer}. The second claim is immediate since the minimal
Weierstrass model of $E$ is given by
$y^2=(x-\beta_1)(x-\beta_2)(x-\beta_3)$, where the $\beta_i$ are the
conjugates of an absolutely bounded rational multiple of $\alpha$.
\end{proof}

We now prove the following result.
\begin{proposition}\label{prop:dydy}
  For positive real numbers $X$ and $Q\leq X$, we have
  \begin{equation}\label{eqeezsfprop}
    \bigl|\bigl\{(E,\eta):E\in\E_\sf,\;\eta\in\Sel_2(E),
    \;X<C(E)\leq 2X,\;QX<\Delta(E)\leq 2QX\bigl\}\bigl|\ll_\epsilon X^{5/6+\epsilon}/Q^{1/6}.
  \end{equation}
  where the implied constant is independent of $X$ and $Q$.
\end{proposition}
\begin{proof}
Let $E\in\E_\sf$ be an elliptic curve satisfying the conductor and
discriminant bounds of \eqref{eqeezsfprop}, and let
$\sigma(E)=(K,\alpha)$. It is easy to verify from Table \ref{tabloc}
that $\Delta(K) = C(E)^2/\Delta(E)$.  Therefore, it follows that
$X/(2Q)<\Delta(K)\leq 4X/Q$, and that $|\alpha|\ll
H(E)^{1/6}\ll_\epsilon(QX)^{1/6+\epsilon}$.

Since the map $\sigma$ is injective, it follows that
the left hand side of \eqref{eqeezsfprop} is
\begin{equation}\label{eqprlemsfcf}
  \ll_\epsilon X^\epsilon\sum_{\substack{[K:\Q]=3\\\frac{X}{2Q}<\Delta(K)\leq \frac{4X}{Q}}}
  N'_K((QX)^{1/6+\epsilon})\,|\Cl(K)[2]|,
\end{equation}
where $N'_K(Y)$ denotes the number of primitive elements $\alpha$ in
$\cO_K^{\Tr =0}$ such that $|\alpha|<Y$ and the pair $(K,\alpha)$ is in
the image of $\sigma$. We now split the above sum over cubic algebras
$K$ into three parts, corresponding to the sizes $\ell_1(K)$ and
$\ell_2(K)$ of the successive minima of $\cO_K^{\Tr = 0}$.

First, if $(QX)^{1/6+\epsilon}\ll\ell_1(K)$, then the contribution to
\eqref{eqprlemsfcf} is $0$. Second, assume that $\ell_2(K)\ll
(QX)^{1/6+\epsilon}$. Then Lemma \ref{lemkalcountel} yields the bound
\begin{equation*}
N_K'((QX)^{1/6+\epsilon})\ll_\epsilon (QX)^{1/3+\epsilon}/\sqrt{X/Q}\ll \frac{Q^{5/6}}{X^{1/6-\epsilon}}.
\end{equation*}
Using Bhargava's result \cite[Theorem 5]{manjulcountquartic} to bound
the sum of $|\Cl(K)[2]|$ over cubic fields $K$ with the prescribed
discriminant range, and using the well known genus-theory bounds
$\Cl(K)[2]\ll|\Delta(K)|^\epsilon$, for each reducible cubic $K$, we obtain:
\begin{equation*}
\sum_{\substack{[K:\Q]=3\\\frac{X}{2Q}<\Delta(K)\leq \frac{4X}{Q}\\\ell_2(K)\ll (QX)^{1/6+\epsilon}}}
  N'_K((QX)^{1/6+\epsilon})\,|\Cl(K)[2]|\ll_\epsilon X^\epsilon\cdot \frac{Q^{5/6}}{X^{1/6-\epsilon}}\cdot \frac{X}{Q}=\frac{X^{5/6+2\epsilon}}{Q^{1/6}}.
\end{equation*}

Finally, we bound the contribution of cubic \'{e}tale algebras $K$ such that
$\ell_1(K)\ll (QX)^{1/6+\epsilon}\ll\ell_2(K)$. In this case, we have
$N_K'((QX)^{1/6+\epsilon})\leq 1$ and
\begin{equation*}
  \sk(K)=\ell_2(K)/\ell_1(K)\gg \sqrt{\Delta(K)}/\ell_1(K)^2
  \gg X^{1/6}/Q^{5/6}.
\end{equation*}
Suppose first $K = \Q\oplus L$ is reducible. Then $\cO_K^{\Tr=0}$ has
an integral basis given by $\{(-2,1),(0,\sqrt{d})\}$ where $\Delta(K)
= d$ or $4d$. When $d$ is small, say bounded by $100$, we get an
$O(1)$ contribution to \eqref{eqprlemsfcf}. When $d$ is large, the
above basis is a Minkowski basis and $(-2,1)$ is the smallest, and
hence unique, primitive traceless element. However, this point does
not correspond to an elliptic curve since the corresponding cubic
polynomial is $(x-2)(x+1)^2$ which has a double root. Hence, we get no
contribution in this case. It remains to consider the case where $K$
is a cubic field. Applying Theorem \ref{thm:skew}, we obtain a bound
of
$$
O_\epsilon\bigl(X^\epsilon(X/Q)/(X^{1/6}/Q^{5/6})\bigr)=
O_\epsilon(X^{5/6+\epsilon}/Q^{1/6}),
$$
on the contribution to \eqref{eqprlemsfcf} over cubic fields $K$ with $\ell_1(K)\ll (QX)^{1/6+\epsilon}\ll\ell_2(K)$, as desired.
\end{proof}

\vspace{.05in}
\noindent\textbf{Proof of Theorem \ref{thunifsqi}:} Note if the
conductor $C(E)$ is bounded by $X$ and the index $\Delta(E)/C(E)$ is
squarefree, then the index is also bounded by $X$. Divide the
conductor range $[1,X]$ into $\log X$ dyadic ranges, and for each such
range divide the index range $[M,X]$ into $\log X$ dyadic ranges, and
then apply Proposition~\ref{prop:dydy} on each pair of dyadic
ranges. Theorem \ref{thunifsqi} follows. $\Box$

\subsection{The family of elliptic curves with bounded index}

As in \S3, let $\Sigma$ be a finite set of pairs $(p,T_p)$, where $p$
is a prime number and $T_p=\I\I\I$, $\I\rV$, or $\I_{\geq 2}$ is a
Kodaira symbol. Recall the invariants $Q(\Sigma)$,
$m_\odd(\Sigma)$ and $m_T(\Sigma)$ for Kodaira symbols $T$. We further
define $m_\even(\Sigma)$ to be the product of $p$ over pairs
$(p,\I_{2k})$ in $\Sigma$. We define $\E(\Sigma)$ to be the set of
elliptic curves $E\in \E$ such that the Kodaira symbol at $p$ of $E$
is $T_p$ for every pair $(p,T_p)\in \Sigma$.
Given a set of five positive real numbers
$$
S=\{m_{\I\I\I},m_{\I\rV},m_{\even},m_{\odd},Q\},
$$ we let $\E(S)$ denote the set of elliptic curves $E$ such that the
product $P$ of primes at which $E$ has Kodaira symbol $\I\I\I$
(resp.\ $\I\rV$, $I_{2(k\geq 1)}$, $I_{2(k\geq 1)+1}$) satisfies
$m_{\I\I\I}\leq P<2 m_{\I\I\I}$ (resp.\ $m_{\I\rV}\leq P<2
m_{\I\rV}$, $m_{\even}\leq P<2 m_{\even}$, $m_{\odd}\leq P<2
m_{\odd}$), and $Q\leq Q(E)<2Q$.
The following result is a consequence of Theorems \ref{thm:equimain}
and \ref{qinvmt}.
\begin{proposition}\label{prop45cb}
Let $S=\{m_{\I\I\I}, m_{\I\rV}, m_{\even}, m_{\odd}, Q\}$ be as above
and let $Y$ be a positive real number. Then
\begin{equation}\label{eq45combbound}
\begin{array}{rcl}
&&\displaystyle\#\{E\in \E(S):|\Delta(E)|<Y\}\\[.1in]
&\ll_\epsilon&\displaystyle
Y^\epsilon\,\min\Bigl(
\frac{Y^{5/6}m_{\even}}{Q^2m_{\I\rV}}+\frac{Qm_{\I\I\I}m_{\even}m_{\odd}^2}{Y^{1/6}},
\frac{Y^{5/6}}{Qm_{\I\I\I}m_{\I\rV}m_{\odd}}+\frac{Y^{7/12}}{Q^{1/2}}
\Bigr).
\end{array}
\end{equation}
\end{proposition}
\begin{proof}
First note that if $E\in\E$, then $H(E)\ll
\Delta(E)^{1+\epsilon}$ from the $j$-invariant bound.
It is enough to prove that the left hand side of \eqref{eq45combbound}
is bounded (up to a factor of $Y^\epsilon$) by both terms in the
minimum.  For the second term, this is a direct consequence of Theorem
\ref{qinvmt} and Table \ref{tabloc}.

For the first term, note that the set of monic cubic polynomials
corresponding to curve in $\E(S)$ is clearly the union of
$O_\epsilon(Y^\epsilon m_{\I\I\I}m_{\I\rV}m_{\even}m_{\odd})$ sets
$U_0(\Z)_\Sigma$, where each such $\Sigma$ satisfies
$m_{\I\I\I}(\Sigma)\sim m_{\I\I\I}$, $m_{\I\rV}(\Sigma)\sim
m_{\I\rV}$, $m_{\even}(\Sigma)\sim m_{\even}$, $m_{\odd}(\Sigma)\sim
m_{\odd}$, and $Q(\Sigma)\sim Q$. In \S3, we obtained bounds on the
number of elements in $U(\Z)_\Sigma$ with height bounded by $Y$. Since
the set $U(\Z)_\Sigma$ is invariant under the linear $\Z$-action, we
have
\begin{equation*}
|\{f\in U_0(\Z)_\Sigma:H(f)<Y\}|\ll Y^{-1/6}|\{f\in U(\Z)_\Sigma:H(f)<Y\}|.
\end{equation*}
Combining this with Theorem \ref{thm:equimain}, and multiplying with
the number of different $\Sigma$'s required to cover the set $\E(S)$,
we obtain the result.
\end{proof}

\noindent\textbf{Proof of Theorem \ref{thunifsmind}:} Given positive
real numbers $X$ and $Y$, let $\E(S;X,Y)$ denote the set of
$E\in\E(S)$ that satisfy $X\leq C(E)<2X$, and $Y\leq\Delta(E)<2Y$.
Fix constants $0<\kappa<7/4$ and $0<\delta$.
We first obtain bounds on the sizes of the sets $\E(S;X,Y)$.  Let
$E$ be an elliptic curve in $\E(S;X,Y)$, and let $P$ be the
contribution to the conductor of $E$ that is prime to
$m_{\I\I\I}m_{\I\rV}m_{\even}m_{\odd}$. Then we have by Table \ref{tabloc}
\begin{equation*}
\begin{array}{rcccl}
X&\asymp&C(E)&\asymp& m_{\I\I\I}^2m_{\I\rV}^2m_{\even}m_{\odd}P;\\[.1in]
Y&\asymp&\Delta(E)&\asymp& m_{\I\I\I}m_{\I\rV}^2m_{\odd}Q^2P.
\end{array}
\end{equation*}
Therefore, in order for $\E(S;X,Y)$ to be nonempty, we must have
\begin{equation}\label{eqQXY}
\frac{Y}{Q^2}\asymp\frac{X}{m_{\I\I\I}m_{\even}}.
\end{equation}
First note that we have
\begin{equation}\label{eq6mtb}
 \frac{Y^{5/6}m_{\even}}{Q^2m_{\I\rV}}\ll \frac{X}{Y^{1/6}}.
\end{equation}
Moreover,
\begin{equation}\label{eq6atb}
\begin{array}{rl}
&\displaystyle \min\Bigl(\frac{Qm_{\I\I\I}m_{\even}m_{\odd}^2}{Y^{1/6}}, \frac{Y^{5/6}}{Qm_{\I\I\I}m_{\I\rV}m_{\odd}}\Bigr)
\leq \Bigl(\frac{Y^{15/6-1/6}m_{\even}}{Q^2m_{\I\I\I}^2m_{\I\rV}^3m_{\odd}}\Bigr)^{1/4}\ll X^{1/4}Y^{1/3};\\[.2in]
&\displaystyle \min\Bigl(\frac{Qm_{\I\I\I}m_{\even}m_{\odd}^2}{Y^{1/6}}, \frac{Y^{7/12}}{Q^{1/2}}\Bigr)
\leq(Ym_{\I\I\I}m_{\even}m_{\odd}^2)^{1/3}\ll\frac{Y^{2/3}}{X^{1/3}}.
\end{array}
\end{equation}

Assume that $Y$ satisfies the bound $X^{1+\delta}\ll Y\ll X^\kappa$
for $\delta>0$ and $\kappa<7/4$. Proposition \ref{prop45cb},
\eqref{eq6mtb}, and \eqref{eq6atb} imply that we have
\begin{equation}\label{eq6ftb}
|\E(S;X,Y)|\ll_\epsilon X^{5/6-\theta+\epsilon},
\end{equation}
for some positive constant $\theta$ depending only on $\delta$ and
$\kappa$.  It is clear that the set
$$\{E\in\E_\kappa:C(E)<X,|\Delta(E)|>C(E)X^\delta\}$$ is the union of
$O_\epsilon(X^\epsilon)$ sets $\E(S;X_1,Y_1)$, with $X_1\leq X$ and
$X_1^{1+\delta}\ll Y_1\ll X_1^\kappa$. Theorem \ref{thunifsmind} now
follows from \eqref{eq6ftb}. $\Box$


\subsection{Additional uniformity estimates}
We will also need (albeit much weaker) estimates on the number of
elliptic curves with bounded height and additive reduction, as well as
on the number of $\PGL_2(\Z)$-orbits on integral binary quartic forms
whose discriminants are divisible by a large square.

We begin with the following result which follows immediately from
the proof of \cite[Proposition 3.16]{bs2sel}.

\begin{proposition}\label{propelemufec}
The number of pairs $(A,B)\in\Z^2$ such that $H(A,B)<X$ and such that
$p^2\mid\Delta(A,B)$ is $O(X^{5/6}/p^{3/2})$, where the implied
constant is independent of $X$ and $p$.
\end{proposition}

Next, we have the following estimate which is proved in
\cite{bswglobal2}

\begin{proposition}\label{propelemufbqf}
The number of $\PGL_2(\Z)$-orbits $f$ on $V_4(\Z)$ such that $H(f)<X$
and $n^2\mid\Delta(f)$ for some $n>M$ is bounded by
$$O\Bigl(\frac{X^{5/6}}{M^{1-\epsilon}}+X^{19/24+\epsilon}\Bigl),$$
where the error terms are independent of $X$ and $M$.
\end{proposition}
\begin{proof}
This is proved in \cite{bswglobal2}, so we merely give a sketch of the
proof. The idea is to embed the space of integral binary quartic forms
into the space $W_4(\Z)$ of pairs of integral quaternary quadratic
forms:
\begin{equation*}
\begin{array}{rcl}
\pi: V_4(\Z) &\to &W_4(\Z)\\[.1in]
ax^4+bx^3y+cx^2y^2+dxy^3+ey^4&\mapsto&\left(
\left[ \begin{array}{cccc} 0 & 0 & 0 & 0\\ 0 & 0 & 0 & 1\\ 0 & 0 & 1 & 0\\ 0 & 1 & 0 & 0\end{array} \right],
\left[ \begin{array}{cccc} 1 & 0 & 0 & 0\\ 0 & -24a & 6b & 2c\\ 0 & 6b & -4c & -3d\\ 0 & 2c & -3d & -6e\end{array} \right]
\right).
\end{array}
\end{equation*}
Under this map, the cubic resolvents of $f$ and $\pi(f)$ are the same,
and hence $f$ and $\pi(f)$ have the same height and discriminant. We
also note that $\pi$ has an algebraic interpretation: the
$\PGL_2(\Z)$-orbit of a nondegenerate element in $V_4(\Z)$ with cubic
resolvent $g(x)$ corresponds to an element in $H^1(\Q,E_g[2])$, while
the $\GL_2(\Z)\times\GL_4(\Z)$-orbit of an element in $W_4(\Z)$
corresponds with cubic resolvent $g(x)$ corresponds to an element in
$H^1(\Q,E_g[4])$. Then the map $\pi$ simply corresponds to natural
map $$H^1(\Q,E_g[2])\rightarrow H^1(\Q,E_g[4]).$$

As proven in \cite{bs4sel}, every element in $W_4(\Z)$ having integral
coefficients and discriminant divisible by $n^2$, for some squarefree
integer $n$, is $\GL_2(\Q)\times\GL_4(\Q)$-equivalent to some element
in $W_4(\Z)$ whose discriminant is divisible by $n^2$ for mod $n$
reasons (in the terminology of \cite{manjul-geosieve}). Then an
application of the Ekedhal sieve in conjunction with
geometry-of-numbers methods counting $\GL_2(\Z)\times\GL_4(\Z)$-orbits
on $W_4(\Z)$, yields the result.
\end{proof}

\section{Asymptotics for families of elliptic curves}

Let $p$ be a fixed prime. An elliptic curve $E$ over $\Q$ has either
good reduction, multiplicative reduction, or additive reduction at
$p$. For every prime $p\geq 5$, let $\Sigma_p$ be a nonempty subset of
possible reduction types. We say that $\Sigma=(\Sigma_p)_p$ is a {\it
  collection of reduction types} and that such a collection is {\it
  large} if for all large enough primes $p$, the set $\Sigma_p$
contains at least the good and multiplicative reduction types.

For a large collection $\Sigma$, let $\E_\sf(\Sigma)$
(resp.\ $\E_\kappa(\Sigma)$) denote the set of elliptic curves
$E\in\E_\sf$ (resp.\ $E\in\E_\kappa$) such that for all primes $p\geq
5$, the reduction type of $E$ at $p$ belongs to $\Sigma_p$. In this
section, we prove the following theorem, from which Theorems
\ref{thmmain} and \ref{thmsel} immediately follow.

\begin{theorem}\label{thmmainlarge}
Let $\Sigma$ be a large collection of elliptic curves. Let
$\kappa<7/4$ be a positive constant. Then we have
\begin{equation}\label{eqEFSig}
\begin{array}{rcl}
\displaystyle\#\{E\in \E_\sf(\Sigma)^\pm:\; C(E)<X\}&\sim&
\displaystyle\frac{\alpha^\pm}{60\sqrt{3}}
\frac{\Gamma(1/2)\Gamma(1/6)}{\Gamma(2/3)}
\prod_p\bigl(c_g(p)e_g(p)+c_m(p)e_m(p)+c_a(p)e_a(p)\bigr)\cdot X^{5/6},\\[.2in]
\displaystyle\#\{E\in \E_\kappa(\Sigma)^\pm:\; C(E)<X\}&\sim&
\displaystyle \frac{\alpha^\pm}{60\sqrt{3}}
\frac{\Gamma(1/2)\Gamma(1/6)}{\Gamma(2/3)}
\prod_p\bigl(c_g(p)f_g(p)+c_m(p)f_m(p)+c_a(p)f_a(p)\bigr)\cdot X^{5/6},
\end{array}
\end{equation}
where $\alpha^+=1$, $\alpha^-=\sqrt{3}$, $c_g(p)$
$($resp.\ $c_m(p)$, $c_a(p))$ is $1$ or $0$ depending on whether
$\Sigma_p$ contains the good $($resp.\ multiplicative, additive$)$
reduction type, and $e_*(p)$ and $f_*(p)$ are given by
$$
e_g(p):= \displaystyle 1-\frac{1}{p};\qquad e_m(p):=\frac{1}{p}\Bigl(1+\frac{1}{p^{1/6}}\Bigr)\Bigl(1-\frac{1}{p}\Bigr)^2;\qquad e_a(p):=\displaystyle\frac{1}{p^2}\Bigl(1+\frac{1}{p^{1/6}}\Bigr)\Bigl(1-\frac{1}{p}\Bigr);$$
$$f_g(p):= \displaystyle 1-\frac{1}{p};\qquad
\displaystyle f_m(p):=\frac{1}{p}\Bigl(1-\frac{1}{p^{1/6}}\Bigr)^{-1}\Bigl(1-\frac{1}{p}\Bigr)^2;$$
$$\displaystyle  f_a(p):=
\displaystyle \frac{1}{p^{5/3}}\Bigl(1-\frac{1}{p}\Bigr)\Bigl(1+\frac{1}{p^{1/6}}+\frac{1}{p^{7/6}}\Bigr) +
\displaystyle\frac{1}{p^2}\Bigl(1-\frac{1}{p}\Bigr)\Bigl(1-\frac{1}{p^{1/6}}\Bigr)^{-1}
\Bigl(3-\frac{2}{p^{1/2}}\Bigr).$$
Furthermore, when elliptic curves in $\E_\sf(\Sigma)$ are ordered by
conductor, the average size of their $2$-Selmer groups is $3$.
\end{theorem}

\subsection{The family $\E$ ordered by discriminant}

We write elliptic curves $E\in\E$ in their minimal short Weierstrass
model. In this case, it is easy to check that we have
\begin{equation*}
\E=\bigl\{E_{AB}:j(E_{AB})<\log(\Delta(E_{AB})),\;\;
16\mid A,\;B\equiv 16\pmod{64},\;\;
3\nmid A\bigr\}
\end{equation*}
Moreover, for every $E_{AB}\in\E$, we have
$\Delta(E_{AB})=\Delta(A,B)/2^8$.  To count elements in $\E$ with
bounded discriminant, we need to incorporate the bound
$j(E_{AB})<\log\Delta(E_{AB})$, which is not a semialgebraic condition
in $A$ and $B$. However it is clearly definable in an o-minimal
structure. Hence we use the following result of Barroero--Widmer
\cite[Theorem 1.3]{ominimal}.
\begin{theorem}\label{thomin}
  Let $m$ and $n$ be positive integers, let $\Lambda\subset\R^n$ be a
  lattice and denote the successive minima of $\Lambda$ by
  $\lambda_i$.  Let $Z\subset\R^{m+n}$ be a definable family, and
  suppose the fibers $Z_T$ are bounded. Then there exists a constant
  $c_Z\in \R$, depending only on the family $Z$, such that
\begin{equation*}
  \Bigl|\#\bigl(Z_T\cap\Lambda\bigr)-\frac{\Vol(Z_T)}{\det(\Lambda)}\Bigr|
  \leq c_Z\sum_{j=0}^{n-1}\frac{V_j(Z_T)}{\lambda_1\cdots\lambda_j},
\end{equation*}
where $V_j(Z_T)$ is the sum of the $j$-dimensional volumes of the
orthogonal projections of $Z_T$ on every $j$-dimensional coordinate
subspace of $\R^n$.
\end{theorem}

For a pair $(A,B)\in\R^2$ with $\Delta(A,B)\neq 0$, let $j(A,B)$
denote $j(E_{AB})$.  For any set $S\subset\Z^2$ defined by congruence
conditions, let $\nu(S)$ denote the volume of the closure of
$S$ in $\hat{\Z}^2$. Equivalently, $\nu(S)$ is the product over the
primes $p$ of the closure of $S$ in $\Z_p^2$. We have the following
immediate consequence of Theorem \ref{thomin}.

\begin{proposition}\label{eqabcountdisc}
Let $\Lambda\subset\Z^2$ denote a set of pairs $(A,B)$ defined by
congruence conditions on $A$ and $B$ modulo some positive integer
$n<X^{1/3-\epsilon}$. Then we have
\begin{equation*}
\#\bigl\{(A,B)\in\Lambda: j(A,B)<\log(2^{-8}\Delta(A,B)),\;0<\pm\Delta(A,B)<X\bigr\}
=\nu(\Lambda)c_\infty^\pm(X)+O_\epsilon(X^{1/2+\epsilon}),
\end{equation*}
where $c_\infty^\pm(X)$ denotes the volume of the set
\begin{equation*}
C^\pm(X):=\bigl\{(A,B)\in\R^2: j(A,B)<\log(2^{-8}\Delta(A,B)),\;0<\pm\Delta(A,B)<X\bigr\}
\end{equation*}
computed with respect to Eucledean measure normalized so that $\Z^2$
has covolume $1$.
\end{proposition}

Since the set $\E$ arises by imposing congruence conditions modulo
infinitely many primes, we use a simple sieve to determine asymptotics
for the number of elliptic curves in $\E$ with bounded discriminant.
\begin{theorem}
We have
\begin{equation*}
  \#\bigl\{E\in\E:0<\pm\Delta(E)<X\}\sim \frac{\alpha^\pm}{60\sqrt{3}}\cdot\frac{\Gamma(1/2)\Gamma(1/6)}{\Gamma(2/3)}
  \cdot\prod_{p\geq 5}\Bigl(1-\frac{1}{p^{10}}\Bigr)X^{5/6},
\end{equation*}
where $\alpha^+=1$ and $\alpha^-=\sqrt{3}$.
\end{theorem}
\begin{proof}
First, we describe the set of elliptic curves $E_{AB}:y^2=x^3+Ax+B$
that have good reduction at $2$ and $3$ in Tables \ref{tabgr2} and
\ref{tabgr3}, respectively. In both tables, the first column describes
the congruence conditions on $A$, the second describes congruence
conditions at $B$, the third gives the $2$-part (resp.\ the $3$-part)
of the discriminant $\Delta(A,B)=4A^3+27B^2$, and the fourth column
gives the density of these congruence conditions inside the space
$(A,B)\in \Z_p^2$ for $p=2$ and $3$. Below, $\delta$ is either $0$ or
$1$.
\begin{table}[ht]
\centering
\begin{tabular}{|l | l| c|c|}
\hline
$A$ & $B$ & $\Delta_2$ & Density \\  
\hline
$\equiv 0\pmod{2^4}$ & $\equiv 2^4\pmod{2^6}$ & $2^8$ & $2^{-10}$\\
$\equiv (5+\delta\cdot 2^6)\pmod{2^7}$ & $\equiv (22+\delta\cdot 2^6)\pmod{2^7}$ & $2^8$ & $2^{-13}$\\
$\equiv (13+\delta\cdot 2^6)\pmod{2^7}$ & $\equiv (14+\delta\cdot 2^6)\pmod{2^7}$ & $2^8$ & $2^{-13}$\\
$\equiv (21+\delta\cdot 2^6)\pmod{2^7}$ & $\equiv (38+\delta\cdot 2^6)\pmod{2^7}$ & $2^8$ & $2^{-13}$\\
$\equiv (29+\delta\cdot 2^6)\pmod{2^7}$ & $\equiv (94+\delta\cdot 2^6)\pmod{2^7}$ & $2^8$ & $2^{-13}$\\
$\equiv (37+\delta\cdot 2^6)\pmod{2^7}$ & $\equiv (54+\delta\cdot 2^6)\pmod{2^7}$ & $2^8$ & $2^{-13}$\\
$\equiv (45+\delta\cdot 2^6)\pmod{2^7}$ & $\equiv (46+\delta\cdot 2^6)\pmod{2^7}$ & $2^8$ & $2^{-13}$\\
$\equiv (53+\delta\cdot 2^6)\pmod{2^7}$ & $\equiv (70+\delta\cdot 2^6)\pmod{2^7}$ & $2^8$ & $2^{-13}$\\
$\equiv (61+\delta\cdot 2^6)\pmod{2^7}$ & $\equiv (126+\delta\cdot 2^6)\pmod{2^7}$ & $2^8$ & $2^{-13}$\\
\hline
\end{tabular}
\caption{Elliptic curves $E_{AB}$ with good reduction at $2$}\label{tabgr2}
\end{table}

\begin{table}[ht]
\centering
\begin{tabular}{|c | c| c|c|}
\hline
$A$ & $B$ & $\Delta_3$ & Density \\  
\hline
$3\nmid A$ &- & 1 & $2\cdot 3^{-1}$ \\
$3^4\parallel A$ & $3^6\mid B$ & $3^{12}$ & $2\cdot 3^{-11}$ \\
$\equiv 2\cdot 3^3\pmod{3^6}$ & $\equiv (\pm 20,\pm 34)\cdot 3^3\pmod{3^7}$ & $3^{12}$ & $4\cdot 3^{-13}$ \\
$\equiv 5\cdot 3^3\pmod{3^6}$ & $\equiv (\pm 11,\pm 16)\cdot 3^3\pmod{3^7}$ & $3^{12}$ & $4\cdot 3^{-13}$ \\
$\equiv 8\cdot 3^3\pmod{3^6}$ & $\equiv (\pm 2,\pm 29)\cdot 3^3\pmod{3^7}$ & $3^{12}$ & $4\cdot 3^{-13}$ \\
$\equiv 11\cdot 3^3\pmod{3^6}$ & $\equiv (\pm 7,\pm 20)\cdot 3^3\pmod{3^7}$ & $3^{12}$ & $4\cdot 3^{-13}$ \\
$\equiv 14\cdot 3^3\pmod{3^6}$ & $\equiv (\pm 16,\pm 38)\cdot 3^3\pmod{3^7}$ & $3^{12}$ & $4\cdot 3^{-13}$ \\
$\equiv 17\cdot 3^3\pmod{3^6}$ & $\equiv (\pm 2,\pm 25)\cdot 3^3\pmod{3^7}$ & $3^{12}$ & $4\cdot 3^{-13}$ \\
$\equiv 20\cdot 3^3\pmod{3^6}$ & $\equiv (\pm 7,\pm 34)\cdot 3^3\pmod{3^7}$ & $3^{12}$ & $4\cdot 3^{-13}$ \\
$\equiv 23\cdot 3^3\pmod{3^6}$ & $\equiv (\pm 11,\pm 38)\cdot 3^3\pmod{3^7}$ & $3^{12}$ & $4\cdot 3^{-13}$ \\
$\equiv 26\cdot 3^3\pmod{3^6}$ & $\equiv (\pm 25,\pm 29)\cdot 3^3\pmod{3^7}$ & $3^{12}$ & $4\cdot 3^{-13}$ \\
\hline
\end{tabular}
\caption{Elliptic curves $E_{AB}$ with good reduction at $3$}\label{tabgr3}
\end{table}

We now apply Proposition \ref{eqabcountdisc}. Let $1\leq i\leq 9$ and
$1\leq j\leq 11$ be integers, and consider the set of integers $(A,B)$
that satisfy line $i$ of Table \ref{tabgr2} and line $j$ of Table
\ref{tabgr3}. Let $\nu_{ij}=\nu_2(i)\cdot \nu_3(j)$ denote the
density of this set of integers, and let $\Delta_{ij}=\Delta_2(i)\cdot
\Delta_3(j)$ denote the product of the $2$- and $3$-parts of the
discriminant $\Delta(A,B)$. It is necessary to count the number of
pairs $(A,B)\in\Z^2$ that satisfy the following properties:
\begin{enumerate}
\item The pair $(A,B)$ satisfies the $i$th (resp.\ $j$th) condition of
  Table \ref{tabgr2} (resp.\ Table \ref{tabgr3});
\item $0<\pm\Delta(A,B)<\Delta_{ij}X$;
\item $j(A,B)<\log(\Delta(A,B)/\Delta_{ij})$;
\item for all primes $p\geq5$, either $p^4\nmid A$ or $p^6\nmid B$.
\end{enumerate}
Counting the pairs $(A,B)$ which satisfy the first four properties is
immediate from Proposition \ref{eqabcountdisc}, and the fifth
condition can be imposed by applying a simple inclusion exclusion
sieve. We thus obtain
\begin{equation*}
\#\bigl\{E\in\E:0<\pm\Delta(E)<X\}\sim \sum_{i,j}\nu_{ij}\cdot\prod_{p\geq 5}(1-p^{-10})\cdot c_\infty^\pm(\Delta_{ij}\cdot X).
\end{equation*}
The values $c_\infty^\pm(X)$ scale as follows: we
have $$c_\infty^\pm(X)=X^{5/6}c_\infty^\pm(1)+o(X).$$ Furthermore, the
values $c_\infty^\pm(1)$ are computed in
\cite[\S2]{watkins-heuristics} to be
\begin{equation*}
  c_\infty^+(1)=\frac{2}{4^{1/3}\cdot 27^{1/2}}\cdot\frac{1}{5}\cdot B(1/2,1/6);\quad\quad
  c _\infty^-(1)=\frac{2}{4^{1/3}\cdot 27^{1/2}}\cdot \frac{3}{5}\cdot B(1/2,1/3)=\sqrt{3}\,c_\infty^+(1).
\end{equation*}
Above, $B(x,y)$ denotes the beta function given by
\begin{equation*}
B(x,y)=\int_{0}^1 t^{x-1}(1-t)^{y-1}dt=\frac{\Gamma(x)\Gamma(y)}{\Gamma(x+y)}.
\end{equation*}
We therefore obtain
\begin{equation*}
\begin{array}{rcl}
\displaystyle\#\bigl\{E\in\E:0<\pm\Delta(E)<X\}&\sim&\displaystyle
\sum_{i,j}\nu_{ij}\Delta_{ij}^{5/6}\cdot\prod_{p\geq 5}(1-p^{-10})\cdot c_\infty^\pm(1)\cdot X^{5/6}
\\&=&\displaystyle
c_\infty^\pm(1)\Bigl(\sum_i\nu_2(i)\Delta_2(i)^{5/6}\Bigr)
\Bigl(\sum_i\nu_3(i)\Delta_3(i)^{5/6}\Bigr)
\prod_{p\geq 5}(1-p^{-10})\cdot X^{5/6}
\\&=&\displaystyle
\frac{2^{2/3}}{4}\frac{2\alpha^\pm}{4^{1/3}\cdot 3^{3/2}\cdot 5}
\frac{\Gamma(1/2)\Gamma(1/6)}{\Gamma(2/3)}
\prod_{p\geq 5}(1-p^{-10})\cdot X^{5/6}
\\&=&\displaystyle
\frac{\alpha^\pm}{60\sqrt{3}}
\frac{\Gamma(1/2)\Gamma(1/6)}{\Gamma(2/3)}
\prod_{p\geq 5}(1-p^{-10})\cdot X^{5/6},
\end{array}
\end{equation*}
as necessary
\end{proof}

\subsection{Ordering elliptic curves by conductor}


Suppose that $\GG$ is equal to $\E_*(\Sigma)$ for a large collection
of reduction types $\Sigma$, where $*$ is either $\sf$ or some positive
$\kappa<7/4$. Pick a small positive constant $\delta<1/9$. Then there
exists a positive constant $\theta$ such that

\begin{equation}\label{eqinexec}
\begin{array}{rcl}
\displaystyle\#\{E\in \GG^\pm:C(E)<X\}&=&
\displaystyle\sum_{n\geq 1}
\#\{E\in \GG^\pm:\ind(E)=n;\;\Delta(E)<nX\}\\&=&
\displaystyle\sum_{n,q\geq 1}\mu(q)\,
\#\{E\in \GG^\pm:nq\mid\ind(E);\;\Delta(E)<nX\}\\&=&
\displaystyle\sum_{\substack{n,q\geq 1\\nq<X^\delta}}\mu(q)\,
\#\{E\in \GG^\pm:nq\mid\ind(E);\;\Delta(E)<nX\}+O(X^{5/6-\theta})
\end{array}\end{equation}
where we bound the tail using the uniformity estimates in Theorems
\ref{thunifsqi} and \ref{thunifsmind}. We perform another inclusion
exclusion sieve to evaluate each summand of the right hand side of the
above equation: for each prime $p$, let
$\chi_{\Sigma_p,nq}:\Z_p^2\to\R$ denote the characteristic function of
the set of all $(A,B)\in\Z_p^2$ that satisfy the reduction type
specified by $\Sigma_p$ and satisfy $nq\mid\ind(E_{AB})$. Let $\chi_p$
denote $1-\chi_{\Sigma_p,nq}$, and define $\chi_k:=\prod_{p\mid
  k}\chi_p$ for squarefree integers $k$. Then we have
\begin{equation*}
\prod_{p}\chi_{\Sigma_p,nq}(A,B)=\sum_k\mu(k)\chi_k(A,B)
\end{equation*}
for every $(A,B)\in\Z^2$. Set $\nu_*(nq,\Sigma)$ to be
the product over all primes $p$ of the integral of
$\chi_{\Sigma_p,nq}$. Therefore, for $nq<X^\delta$, we obtain
\begin{equation*}
\begin{array}{rcl}
\displaystyle\#\{E\in \GG^\pm:nq\mid\ind(E);\;\Delta(E)<nX\}&=&
\displaystyle\sum_{\substack{(A,B)\in\Z^2\\0<\pm\Delta(E_{AB})<nX}}\sum_{k\geq 1}\mu(k)\chi_k(A,B)\\[.2in]
&=&\displaystyle\sum_{\substack{(A,B)\in\Z^2\\0<\pm\Delta(E_{AB})<nX}}\sum_{k=1}^{X^{4\delta}}\mu(k)\chi_k(A,B)
+O\Bigl(\frac{(nX)^{5/6}}{X^{2\delta}}\Bigr)\\[.4in]&=&
\displaystyle c_\infty^\pm(nX)\nu_*(nq,\Sigma)+O_\epsilon(X^{1/2+2\delta+\epsilon}+X^{5/6-7\delta/6}),
\end{array}
\end{equation*}
where the second equality follows from the uniformity estimate in
Proposition \ref{propelemufec}, and the third follows from Proposition
\ref{eqabcountdisc} and adding up the volume terms by simply reversing
the inclusion exclusion sieve. Note that the constant $\delta$ has been
specifically picked to be small enough so that Proposition
\ref{eqabcountdisc} applies.

For each $n$, let $\lambda_*(n,\Sigma)$ denote the volume of the
closure in $\hat{\Z}^2$ of the set of all $(A,B)\in\Z^2$ such that
$E_{AB}$ belongs to $\GG=\E_*(\Sigma)$ and $E_{AB}$ has index
$n$. Returning to \eqref{eqinexec}, we obtain
\begin{equation*}
\begin{array}{rcl}
\displaystyle\#\{E\in \GG^\pm:C(E)<X\}&=&
\displaystyle c^\pm_\infty(1)X^{5/6}
\sum_{\substack{n,q\geq 1\\nq<X^\delta}}\mu(q)n^{5/6}\nu_*(nq,\Sigma)
+o(X^{5/6})\\[.2in]
&=&\displaystyle c_\infty^\pm(1)X^{5/6}\sum_{n\geq 1}n^{5/6}\lambda_*(n,\Sigma),
\end{array}
\end{equation*}
where again, the final equality follows by reversing
the inclusion exclusion sieve of \eqref{eqinexec}.

For each prime $p$ and $k\geq 0$, let $\bar{\nu}_*(p^k,\Sigma)$ denote
the $p$-adic density of the set of all $(A,B)\in\Z^2$ such that
$E_{AB}\in\E_*(\Sigma)$ and $\ind_p(E_{AB})=p^k$. The constant
$\lambda_*(n,\Sigma)$ is a product over all $p$ of local densities:
\begin{equation*}
\begin{array}{rcl}
\displaystyle \lambda_*(n,\Sigma)&=&
\displaystyle\prod_{p\nmid n}\bar{\nu}_*(p^0,\Sigma)
\prod_{\substack{p^k\parallel n\\k\geq 1}}\bar{\nu}_*(p^k,\Sigma)\\[.2in]
&=&\displaystyle\prod_p\bar{\nu}_*(p^0,\Sigma)
\prod_{\substack{p^k\parallel n\\k\geq 1}}
\frac{\bar{\nu}_*(p^k,\Sigma)}{\bar{\nu}_*(p^0,\Sigma)}.
\end{array}
\end{equation*}
Hence $\lambda_*(n,\Sigma)$ is a multiplicative function in
$n$, and we have
\begin{equation*}
\begin{array}{rcl}
\displaystyle\sum_{n\geq 1}n^{5/6}\lambda_*(n,\Sigma)&=&
\displaystyle\prod_p\bar{\nu}_*(p^0,\Sigma)
\prod_{p}\Bigl(\sum_{k=0}^\infty p^{5k/6}
\frac{\bar{\nu}_*(p^k,\Sigma)}{\bar{\nu}_*(p^0,\Sigma)}\Bigr)\\[.2in]&=&
\displaystyle\prod_p\Bigl(\sum_{k=0}^\infty p^{5k/6}\bar{\nu}_*(p^k,\Sigma)
\Bigr).
\end{array}
\end{equation*}

The values of $\bar{\nu}_*(p^k,\Sigma)$ are easily computed from Table
\ref{tabden}. We then have \eqref{eqEFSig}, proving the first part of Theorem \ref{thmmainlarge}.

\subsection{The average size of the $2$-Selmer groups of elliptic curves in $\E_\sf(\Sigma)$}

Let $\Sigma$ be a large collection of reduction types. For a positive
integers $n$ and a positive real number $X$, let $\E(\Sigma,n,X)$ denote the
set of elliptic curves $E\in\E_\sf(\Sigma)$, such that $n\mid\ind(E)$
and $|\Delta(E)|<X$. Then, as in the previous subsection, we have
\begin{equation*}
\begin{array}{rcl}
\displaystyle\sum_{\substack{E\in\E(\Sigma)^\pm\\C(E)<X}}(|\Sel_2(E)|-1)&=&
\displaystyle\sum_{n,q\geq 1}\mu(q)\sum_{E\in\E(\Sigma,nq,nX)^\pm}(|\Sel_2(E)|-1)
\\[.2in]&=&
\displaystyle\sum_{\substack{n,q\geq 1\\nq<X^\theta}}\mu(q)\sum_{E\in\E(\Sigma,nq,nX)^\pm}(|\Sel_2(E)|-1)+O_\epsilon(X^{5/6-\theta/6+\epsilon}),
\end{array}
\end{equation*}
for every $\theta>0$, where the second equality is a consequence of
Theorem \ref{thunifsqi}. Therefore, the final assertion of Theorem
\ref{thmmainlarge} follows immediately from the following result.

\begin{proposition}\label{propfinalsel}
There exist positive constants $\theta$ and $\theta_1$ such that
\begin{equation*}
\sum_{E\in\E(\Sigma,qn,nX)^\pm}(|\Sel_2(E)|-1)=2|\E(\Sigma,nq,nX)^\pm|+O(X^{5/6-\theta_1}),
\end{equation*}
for every $nq<X^{\theta}$.
\end{proposition}

Given the uniformity estimate Proposition \ref{propelemufbqf} that we
have already proved, the proof of Proposition \ref{propfinalsel} very
closely follows the proof of \cite[Theorem 3.1]{bs2sel}. In what
follows, we briefly sketch the proof of Theorem \ref{thmmainlarge},
indicating the change needed at the places where it differs from
\cite{bs2sel}. The starting point of the proof is the following
parametrization of the $2$-Selmer groups of elliptic curves in terms
of orbits on integral binary quartic forms. This correspondence is due
to Birch and Swinnerton-Dyer, and we state it in the form of
\cite[Theorem 3.5]{bs2sel}.

\begin{theorem}
Let $E:y^2=x^3+Ax+B$ be an elliptic curve over $\Q$, and set
$I=I(E):=-3A$ and $J=J(E):=-27B$.  Then there is a bijection between
$\Sel_2(E)$ and the set of $\PGL_2(\Q)$-equivalence classes of locally
soluble integral binary quartic forms with invariants $2^4I$ and
$2^6J$.

Moreover, the set of integral binary quartic forms that have a
rational linear factor and invariants equal to $2^4I$ and $2^6J$ lie
in one $\PGL_2(\Q)$-equivalence class, and this class corresponds to
the identity element in $\Sel_2(\Q)$.
\end{theorem}

The second step in the proof is to obtain asymptotics for the number
of $\PGL_2(\Z)$-orbits on the set of integral binary quartic forms
whose coefficients satisfy congruence conditions modulo some small
number $n$, where these forms have bounded invariants. In
\cite{bs2sel}, the invariants were bounded by height. Here instead, we
bound their discriminants and corresponding $j$-invariant: for an
element $f\in V_4(\R)$ with $\Delta(f)\neq 0$, define $j(f)$ to be
$j(E)$ with $E$ given by
\begin{equation*}
E:y^2=x^3-(I/3)x-J/27.
\end{equation*}
For any $\PGL_2(\Z)$-invariant set $S\subset
V_4(\Z)$, let $N^{(i)}_4(S;X)$ denote the number the number of
$\PGL_2(\Z)$-orbits on integral elements $f\in S\subset
V_4^{(i)}(\Z)$, that do not have a linear factor over $\Q$, and
satisfy $\Delta(f)<X$ and $j(f)<\log\Delta(f)$.

In \S5, we defined the sets $R^{(i)}$ which are fundamental sets for
the action of $\PGL_2(\R)$ on $V(\R)^{(i)}$. Then $R^{(i)}$ contains
one element $f\in V(\R)^{(i)}$ having invariants $I$ and $J$, for each
$(I,J)\in\R^2$ with $4I^3-J^2\in\R_{>0}$ for $i=0,2\pm$ and
$4I^3-J^2\in\R_{<0}$ for $i=1$. Furthermore, the coefficients of such
an $f$ are bounded by $O(H(f)^{1/6})$. Define the sets
\begin{equation*}
R^{(i)}(X):=\{f\in R^{(i)}:0<|\Delta(f)|<X;j(f)<\log\Delta(f)\}.
\end{equation*}
Clearly, if $f\in R^{(i)}(X)$ with $\Delta(f)=X$, then $H(f)\ll
X^{1+\epsilon}$ and so the coefficients of $f$ are bounded by
$O(X^{1/6+\epsilon})$.

Let $\delta=1/18$ be fixed. Let $L\subset V(\Z)$ be a lattice defined
by congruence conditions modulo $n$, where $n<X^\delta$. Denote the
set of elements in $L$ that have no linear factor by $L^\irr$ and
define $\nu(L)$ to be the volume of the completion of $L$ in
$V_4(\hat{\Z})$. Let $G_0\subset\PGL_2(\R)$ be a nonempty bounded open
ball, and set $n_1=2$, $n_0=n_{2\pm}=4$. Identically to \cite[\S2.3]{bs2sel}, it
follows that $N_4^{(i)}(L,X)$ is given by
\begin{equation}\label{eqbqfcount}
\begin{array}{rcl}
\displaystyle N_4^{(i)}(L,X)&=&
\displaystyle\frac{1}{n_i\Vol(G_0)}\int_{\gamma\in\PGL_2(\Z)\backslash\PGL_2(\R)}
\#\bigl\{\gamma G_0\cdot R^{(i)}(X)\cap L^\irr\}d\gamma
\\[.2in]&=&
\displaystyle\frac{1}{n_i}\int_{\gamma\in\PGL_2(\Z)\backslash\PGL_2(\R)}
\nu(L)\Vol(G_0\cdot R^{(i)}(X))d\gamma+O(X^{7/9}),
\end{array}
\end{equation}
where the error term is obtained in a similar manner to
\cite[(18)--(20)]{bs2sel}. There are two differences: first, we use
Theorem \ref{thomin} (instead of Davenport's result stated as
\cite[Proposition 2.6]{bs2sel}) to estimate the number of lattice
points in $\gamma G_0\cdot R^{(i)}(X)$. Second, since we are imposing
congruence conditions on $L$ modulo $n<X^\delta$ with $\delta=1/18$,
we cut off the integral over $\gamma$ when the $t$-coefficient of
$\gamma$ in its Iwasawa coordinate is $\gg X^{1/36}$. That way, the
coefficients of the ball $\gamma G_0\cdot R^{(i)}(X)$ are always
bigger than $n$. The precise values of $\delta=1/18$ and $7/9$, the
exponent of the error term, are not important. 

The third step in the proof is to introduce a bounded weight
  function $m:V_4(\Z)\to\R$, which is the product $m=\prod_p m_p$ of
local weight functions $m_p:V_4(\Z_p)\to\R$, such that for all but negligibly many ($\ll_\epsilon X^{3/4+\epsilon}$) elliptic curves $E_{AB}$, we have
\begin{equation*}
|\Sel_2(E_{AB})|-1=\sum_{f\in \frac{V_4(\Z)_{A,B}}{\PGL_2(\Z)}}m(f),
\end{equation*}
where $f$ is varying over $\PGL_2(\Z)$-orbits on integral binary
quartic forms with no linear factor and invariants $I(f)=-3\cdot 2^4I$
and $J(f)=-27\cdot2^6J$. In our situation, we do not need any changes
to this part of the proof.

The fourth and final part of the proof is to perform a sieve so as to
count $\PGL_2(\Z)$-orbits on integral binary quartic forms with
bounded invariants, so that each form $f$ is weighted by
$m(f)$. Performing a standard inclusion-exclusion sieve using
\eqref{eqbqfcount} together with the uniformity estimate Proposition
\ref{propelemufbqf} and the volume computations of
\cite[\S3.3 and \S3.6]{bs2sel} yields Proposition~\ref{propfinalsel}. This
concludes the proof of Theorem \ref{thmmainlarge}.

\bibliography{Arithbib}
\bibliographystyle{amsplain}

\end{document}